\documentclass[a4paper,oneside,10pt]{article}%
\usepackage{amsmath}
\usepackage{amsfonts}
\usepackage{amssymb}
\usepackage{graphicx}
\usepackage{color}
\usepackage{authblk}
\usepackage{comment}
\usepackage[square,numbers,sort&compress]{natbib}%
\providecommand{\U}[1]{\protect\rule{.1in}{.1in}}

\newtheorem{theorem}{Theorem}[section]

\newtheorem{corollary}[theorem]{Corollary}

\newtheorem{definition}[theorem]{Definition}

\newtheorem{lemma}[theorem]{Lemma}

\newtheorem{proposition}[theorem]{Proposition}
\newtheorem{remark}[theorem]{Remark}

\newenvironment{proof}[1][Proof]{\noindent \textbf{#1.} }{\  $\Box$}
\numberwithin{equation}{section}

\begin{document}

\title{Stochastic wave equations driven by  space-time $G$-white noise under sublinear expectation}

\author[a, b]{Xiaojun Ji\thanks{Corresponding author}}

\affil[a]{Research Center for Mathematics and Interdisciplinary Sciences, Shandong University, Qingdao, Shandong 266237, P. R. China.}
\affil[b]{Frontiers Science Center for Nonlinear Expectations (Ministry of Education), Shandong University, Qingdao, Shandong 266237, P. R. China.}

\affil[ ]{\textit {Email: jixj@sdu.edu.cn}}

\renewcommand*{\Affilfont}{\small\it}

\date{}

\maketitle

{
\begin{abstract}
In this paper, we study one-dimensional semilinear stochastic wave
equations driven by multiplicative space-time $G$-white noise under
sublinear expectation. We establish the existence and uniqueness of
global mild solutions. By developing the quadratic variation structure
of space-time $G$-white noise and establishing a
Burkholder-Davis-Gundy inequality for the associated stochastic
integrals, we derive higher-order moment estimates and prove the
quasi-sure H\"older continuity of the solutions. We further show that
the mild solution satisfies the corresponding weak formulation.

\medskip
\medskip


\medskip
\medskip


\end{abstract}}

\section{Introduction}

Stochastic wave equations are a  fundamental class of stochastic partial
differential equations (SPDEs) of hyperbolic type. They arise naturally when wave
propagation is subject to random forcing or fluctuations and provide
mathematical models for a variety of randomly perturbed wave phenomena, such as  randomly forced elastic systems and wave propagation in
physical and geophysical media \cite{Dalang09}. 

The theory of stochastic wave equations has been developed
extensively within the classical probability framework. Walsh's
martingale measure theory \cite{Walsh} provides a fundamental random field
approach to SPDEs driven by space-time white noise. Dalang
\cite{Dalang} later extended this stochastic integration theory to a broad class
of spatially homogeneous noises and  introduced an integrability condition
that has become central to the random field theory of stochastic heat
and wave equations.  
For stochastic wave equations,
existence, uniqueness, moment estimates, and sample path regularity have
subsequently been studied in various settings; see, among others, Conus and Dalang \cite{ConusDalang},
Dalang and Frangos \cite{DalangFrangos}, Dalang and Mueller
\cite{DalangMueller}, and Sanz-Sol\'e
 and Sarr\`a \cite{Sanz}.
For broader accounts of stochastic heat and wave equations and related
SPDEs, we refer to \cite{Dalang09,DS,DaP-Z,Hai,LR15} and the references therein.

The classical theory is formulated under a prescribed probability
measure, and therefore assumes that the probabilistic law governing the
random environment is known. However, in many applications,  the
underlying probability distribution, and in particular its variance or
volatility, may itself be uncertain. Such uncertainty is conceptually
different from ordinary randomness and is often referred to as
Knightian uncertainty or model uncertainty \cite{Knight,EJ13}. A
probabilistic formulation based on a single reference measure is not well suited to capturing such distributional ambiguity.

To address uncertainty in probability distributions, Peng developed the
nonlinear  expectation theory, and, in particular, introduced
$G$-expectation and $G$-Brownian motion 
\cite{P2007,Peng08,P2019}. The $G$-expectation can be represented by
 a non-dominated family of probability measures and provides a
stochastic calculus in which volatility uncertainty is incorporated into
the probabilistic model itself. Since its
introduction, stochastic analysis under $G$-expectation has developed
rapidly. In particular, stochastic differential equations driven by
$G$-Brownian motion have been studied in \cite{Gao09,HJL},
while backward stochastic differential equations under $G$-expectation
were developed in \cite{HJPS,P2019}. Compared with the finite-dimensional
theory, however, stochastic partial differential equations under
sublinear expectation require additional random field and
infinite-dimensional stochastic integration tools and remain much less
developed.

A fundamental step toward such a theory was taken by Ji and Peng
\cite{JP20}, who introduced $G$-Gaussian random fields, spatial
$G$-white noise, and space-time $G$-white noise, together with the
corresponding stochastic integration framework under sublinear
expectation. 
More recently, Ji and Peng
\cite{JP25} studied  stochastic heat equations driven by
multiplicative space-time $G$-white noises. They established existence and uniqueness of mild solutions, developed
a stochastic Fubini theorem under sublinear expectation, and proved that
the mild solution satisfies the corresponding weak formulation. 
These results provide a natural probabilistic and
analytic foundation for the study of further classes of SPDEs under
distributional uncertainty.

The hyperbolic counterpart, namely stochastic wave equations driven by
space-time $G$-white noise, presents  additional difficulties.
 Unlike the
heat kernel, the wave kernel does not possess a comparable smoothing
effect, 
and its characteristic structure requires estimates that are specific to
hyperbolic equations. 
Moreover, nonlinear expectation prevents one from
freely using several classical arguments based on a single probability
measure. Suitable moment inequalities for stochastic
integrals, deterministic and stochastic convolution estimates, and
quasi-sure continuity criteria are required in order to develop a well-posedness and regularity
theory for stochastic wave equations under sublinear expectation. 

Motivated by these considerations, in this paper we develop a
random field theory for the semilinear stochastic wave equation driven by
space-time $G$-white noise  on a bounded interval:
\begin{equation*}
\begin{cases}
\partial_{tt}u(t,x)-\partial_{xx}u(t,x)
=
b(u(t,x))+a(u(t,x))\dot{\mathbf W}(t,x),
& t>0,\ 0<x<L,\\
  u(0,x)=u_0(x),& 0\le x\le L,\\
  \frac{\partial}{\partial t}u(0,x)=v_0(x),& 0<x<L,\\
  u(t,0)=u(t,L)=0,&t\ge 0,
\end{cases}
\end{equation*}
where $\dot{\mathbf W}$ denotes the generalized derivative of a
space-time $G$-white noise and $a,b:\mathbb R\to\mathbb R$ are globally
Lipschitz continuous. We refer to this equation as the
$G$-stochastic wave equation.

Our first objective is to develop the stochastic integral estimates
needed for the wave equation. Building on the quadratic variation theory
of $G$-Brownian motion \cite{P2019}, we study the quadratic and mutual
variations of spatial projections of space-time $G$-white noise. Using
the representation of the sublinear expectation by a family of
probability measures, we identify the corresponding classical quadratic
variations and brackets under each representing measure. This leads to a
Burkholder-Davis-Gundy (BDG) inequality for stochastic integrals with
respect to space-time $G$-white noise, and in particular to the
fixed-time $\mathbf L_G^p$ estimates needed for the subsequent
 analysis.

We next establish deterministic and stochastic convolution estimates
associated with the Dirichlet wave kernel on $(0,L)$.
These estimates
show that the corresponding convolution operators preserve
$\mathbf S_G^p$ and provide the increment bounds required for both
well-posedness and regularity. 
Under suitable assumptions on the initial  conditions, 
we prove, by Picard iteration on arbitrary finite horizons, the
existence and uniqueness of a global mild solution to the $G$-stochastic wave equation. 
Uniqueness ensures consistency of the finite-horizon solutions and hence
yields a global solution on $[0,\infty)\times[0,L]$.
Moreover, using deterministic Fubini arguments together with the
stochastic Fubini theorem under sublinear expectation from \cite{JP25},
we show that every mild solution also satisfies the corresponding weak
formulation.

Finally, we investigate the path regularity of the solution. If, in
addition,
\[
u_0\in C_0^\gamma([0,L]),\qquad
v_0\in L^\infty(0,L),\qquad
0<\gamma\le1,
\]
we derive higher-order moment estimates for the joint space-time
increments of the solution. A Kolmogorov continuity criterion under capacity then yields a
modification whose sample paths are jointly continuous quasi-surely and
locally H\"older continuous on every finite space-time cylinder.
More precisely, if $\gamma\ge1/2$, every H\"older
exponent $\alpha<1/2$ is admissible, whereas if
$0<\gamma<1/2$, the exponent $\gamma$ is attained. Thus the solution
exhibits the characteristic $1/2$ regularity threshold of the
one-dimensional stochastic wave convolution in the framework of
sublinear expectation.

The remainder of the paper is organized as follows.
Section~2 recalls the basic theory of sublinear expectations, space-time
$G$-white noise, and the associated stochastic integrals.
Section~3 studies the quadratic variation structure of spatial
projections of space-time $G$-white noise and establishes the BDG
inequality. Section~4 develops the Dirichlet wave kernel estimates and
proves the existence and uniqueness of global mild solutions, the
mild-to-weak implication, moment estimates, and quasi-sure H\"older
regularity.

\section{Preliminaries}

In this section, we recall some basic notions and results in the sublinear
expectation theory. More details can be found in {\cite{DHP, JP20, JP25, P2007,  P2019}} and references therein.

\subsection{Basic notions of sublinear expectations}

Let $\Omega$ be a  nonempty set and let $\mathcal{H}$ be a linear space of
real-valued functions on $\Omega$ such that if $X_1$, $\cdots$, $X_d \in
\mathcal{H}$, then $\varphi (X_1, X_2, \cdots, X_d) \in \mathcal{H}$ for each
$\varphi \in C_{l.Lip} (\mathbb{R}^d)$, where $C_{l.Lip} (\mathbb{R}^d)$
denotes the linear space of local Lipschitz functions satisfying that for any $x, y \in
\mathbb{R}^d$,
\begin{equation*}
 | \varphi (x) - \varphi (y) | \leq C_{\varphi}  (1 + |x|^m + |y|^m)  |x -
   y|, \text{for some } C_{\varphi} > 0, m \in \mathbb{N} \text{ depending on }
   \varphi .
\end{equation*}
$\mathcal{H}$ is considered as the space of random variables. For each $X_i \in \mathcal{H}$, $1 \leq i \leq d$, $X = (X_1,
\cdots, X_d)$ is called a
$d$-dimensional random vector, denoted by $X \in \mathcal{H}^d$.

\begin{definition}
  \label{sublinear expectation}A {sublinear expectation}
  $\hat{\mathbb{E}}$ on $\mathcal{H}$ is a functional $\hat{\mathbb{E}} :
  \mathcal{H} \rightarrow \mathbb{R}$ satisfying the following properties: for
  each $X, Y \in \mathcal{H}$,
  \begin{itemize}
    \item[(i)] {Monotonicity:}\quad$\hat{\mathbb{E}} [X] \geq
    \hat{\mathbb{E}} [Y]  \text{ if } X \geq Y$;
    
    \item[(ii)] {Constant preserving:}\quad$\hat{\mathbb{E}} [c] = c
    \text{ for } c \in \mathbb{R}$; 
    
    \item[(iii)] {Sub-additivity:}\quad$\hat{\mathbb{E}}  [X + Y]
    \leq \hat{\mathbb{E}} [X] + \hat{\mathbb{E}} [Y]$;
    
    \item[(iv)] {Positive homogeneity:}\quad$\hat{\mathbb{E}}
    [\lambda X] = \lambda \hat{\mathbb{E}} [X]  \text{ for } \lambda \geq 0$.
  \end{itemize}
  The triplet $(\Omega, \mathcal{H}, \hat{\mathbb{E}})$ is called a
  {sublinear expectation space}. In particular, if only the properties (i) and (ii) are satisfied,
  then $\hat{\mathbb{E}}$ is called a {nonlinear expectation} and  $(\Omega, \mathcal{H}, \hat{\mathbb{E}})$ is called a
  {nonlinear expectation space}.
\end{definition}

Throughout this paper, we are mainly interested in sublinear expectations satisfying
the following regular property:
\begin{itemize}
  \item[(v)] {Regularity:}
  {\hspace{0.17em}}{\hspace{0.17em}}{\hspace{0.17em}} If $\{X_i \}_{i =
  1}^{\infty} \subset \mathcal{H}$ satisfies that $X_i (\omega) \downarrow 0$
  as $i \rightarrow \infty$, for each $\omega \in \Omega$, then
  \begin{equation*}
   \lim_{i \rightarrow \infty}  \hat{\mathbb{E}} [X_i] = 0. 
  \end{equation*}
\end{itemize}

Within the framework of  sublinear expectations,  important definitions, including the distribution and independence of random variables,  are introduced as follows.
\begin{definition}
 For a $d$-dimensional random vector $X$ on the sublinear
  expectation space $(\Omega, \mathcal{H}, \hat{\mathbb{E}})$, its distribution is the 
functional  $\mathbb{F}_X$ on $C_{l.Lip} (\mathbb{R}^d)$ defined by
\begin{equation*} \mathbb{F}_X [\varphi] := \hat{\mathbb{E}} [\varphi (X)], \text{ for
   } \varphi \in C_{l.Lip} (\mathbb{R}^d) .
\end{equation*}
\end{definition}
\begin{definition}
  Two $d$-dimensional random vectors $X_1$ and $X_2$,  defined on sublinear
  expectation spaces $(\Omega_1, \mathcal{H}_1, \hat{\mathbb{E}}_1)$ and
  $(\Omega_2, \mathcal{H}_2, \hat{\mathbb{E}}_2)$, respectively, are
  called {identically distributed}, denoted by $X_1 \overset{d}{=}
  X_2$, if $\mathbb{F}_{X_1} =\mathbb{F}_{X_2}$, i.e.,
  \begin{equation*} \hat{\mathbb{E}}_1 [\varphi (X_1)] = \hat{\mathbb{E}}_2 [\varphi (X_2)], 
     \text{ for } \varphi \in C_{l.Lip} (\mathbb{R}^d) .
  \end{equation*}
\end{definition}

\begin{definition}
 A $d_1$-dimensional random vector $Y$ is said to be
  {independent} of another $d_2$-dimensional random vector $X$ on $(\Omega, \mathcal{H}, \hat{\mathbb{E}})$ if, for each test function $\varphi \in
  C_{l.Lip} (\mathbb{R}^{d_1 + d_2})$,
  \begin{equation*} 
  \hat{\mathbb{E}} [\varphi (X, Y)] = \hat{\mathbb{E}} [\hat{\mathbb{E}}
     [\varphi (x, Y)]_{x = X}] . 
  \end{equation*}
  In particular, let $\bar{X}$ and $X$ be two $d$-dimensional random vectors on $(\Omega,
  \mathcal{H}, \hat{\mathbb{E}})$. $\bar{X}$ is called an {independent copy} of
  $X$ if $\bar{X} \overset{d}{=} X$ and $\bar{X}$ is independent of $X$.
\end{definition}

\begin{remark}
We emphasize that independence under a sublinear expectation is in general not symmetric, namely,  ``$Y$ is independent of $X$'' does not necessarily imply
  that ``$X$ is independent of $Y$'' (See  Peng {\cite{P2019}}).
\end{remark}

\subsection{$G$-white noises under sublinear expectations}
Within the  sublinear expectation framework, essential random fields,  including  $G$-Brownian motion and $G$-white noises, are introduced in \cite{P2007, JP20, P2019}.

\begin{definition}
  A $d$-dimensional random vector $X$ on a sublinear
  expectation space $(\Omega, \mathcal{H}, \hat{\mathbb{E}})$ is called
  (centralized) {$G$-normally distributed} if
  \begin{equation*}
   aX + b \bar{X} \overset{d}{=} \sqrt{a^2 + b^2} X, \text{ for } a, b \geq 0,
  \end{equation*}
  where $\bar{X}$ is an independent copy of $X$.
\end{definition}

\begin{definition}
  A $d$-dimensional random vector $\eta$ on a sublinear
  expectation space $(\Omega, \mathcal{H}, \hat{\mathbb{E}})$ is called
  maximally distributed if there exists a bounded, closed, and convex subset $\Lambda\subset \mathbb{R}^d$ such that
    \begin{equation*}
   \hat{\mathbb{E}}[\varphi(\eta)]=\max_{y\in\Lambda}\varphi(y), \text{ for } \varphi\in C_{l.Lip}(\mathbb{R}^d).
  \end{equation*}
  We denote $\eta\sim\mathcal{M}(\Lambda)$.
\end{definition}

\begin{definition}
  On a sublinear expectation space
  $(\Omega, \mathcal{H}, \hat{\mathbb{E}})$, a $d$-dimensional stochastic process $(B_t)_{t \geq 0}$ with $B_t\in \mathcal{H}^d$ for $t\geq0$  is called a {(symmetric) $G$-Brownian motion} if
  the following properties are satisfied:
  \begin{itemize}
    \item[(i)] $B_0 = 0 ;$
    
    \item[(ii)] For each $t, s \geq 0,$ $B_{t + s} - B_t$ and $B_s$ are
    identically distributed and $B_{t + s} - B_t$ is independent of $(B_{t_1}, \cdots, B_{t_n})$ for any $n \in \mathbb{N}$ and $0 \leq t_1 \leq \cdots \leq t_n \leq t.$
    
    \item[(iii)] $\hat{\mathbb{E}} [B_t] = \hat{\mathbb{E}} [- B_t] = 0$ and $\lim_{t
    \downarrow 0} \hat{\mathbb{E}} [| B_t |^3] t^{- 1} = 0$.
  \end{itemize}
\end{definition}

\begin{definition}
  Let $\Gamma$ be any parameter set. A family of random variables $(W_{\gamma})_{\gamma \in \Gamma}\subset\mathcal{H}$ on $(\Omega, \mathcal{H}, \hat{\mathbb{E}})$ is
  called a $1$-dimensional {$G$-Gaussian random field} if for each $n\in\mathbb{N}$ and
  $\gamma_1, \cdots, \gamma_n \in
  \Gamma$,  $n$-dimensional random vector
  $ (W_{\gamma_1}, \cdots, W_{\gamma_n}) \in
  \mathcal{H}^n$ is $G$-normally distributed.
\end{definition}

 It is worth noting that for a $G$-Brownian motion $(B_t)_{t\geq0}$, $B_t$ is $G$-normally distributed for each given $t>0$. However, the finite-dimensional distributions of $(B_t)_{t\geq0}$ are generally not $G$-normally distributed, implying that $G$-Brownian motion is no longer a  $G$-Gaussian random field.

Let $0 \leq \underline{\sigma}^2 \leq \overline{\sigma}^2<\infty$ be any given numbers, and define  $\mathcal{B}_0 (\mathbb{R}) := \{A
  \in \mathcal{B}(\mathbb{R}) : \lambda_A < \infty\}$, where $\lambda_A$
  denotes the Lebesgue measure of $A \in \mathcal{B} (\mathbb{R})$.  

\begin{definition}
  \label{Gwhitenoise}A
  $1$-dimensional  $G$-Gaussian random field $\mathbb{W} = (\mathbb{W}_A)_{A
  \in \mathcal{B}_0 (\mathbb{R})}$ on $(\Omega, \mathcal{H}, \hat{\mathbb{E}})$ is called a {spatial $G$-white
    noise} with parameters $\underline{\sigma}^2$ and   $\overline{\sigma}^2$ if
  \begin{itemize}
    \item[(i)] For each $A \in \mathcal{B}_0 (\mathbb{R})$,
    $\hat{\mathbb{E}} [\mathbb{W}_A^2] = \overline{\sigma}^2 \lambda_A$, $-
    \hat{\mathbb{E}} [- \mathbb{W}_A^2] = \underline{\sigma}^2 \lambda_A$;
    
    \item[(ii)] For each $A_1, A_2 \in \mathcal{B}_0 (\mathbb{R})$ such that $A_1
    \cap A_2 = \emptyset$,
    \begin{align*}
      & \hat{\mathbb{E}} [\mathbb{W}_{A_1} \mathbb{W}_{A_2}] =
      \hat{\mathbb{E}} [- \mathbb{W}_{A_1} \mathbb{W}_{A_2}] = 0, \\
      & \hat{\mathbb{E}} [(\mathbb{W}_{A_1 \cup A_2} - \mathbb{W}_{A_1} -
      \mathbb{W}_{A_2})^2] = 0. 
    \end{align*}
  \end{itemize}
\end{definition}

Throughout this paper, we consistently fix two parameters  $\underline{\sigma}^2$ and   $\overline{\sigma}^2$. 
Building upon the $G$-Brownian motion and  spatial $G$-white noise, we introduce a  spatio-temporal random field, which is called the space-time $G$-white noise.
Set the parameter set as
\begin{equation*}
 \Gamma = \{ [s, t) \times A : 0 \leq s \leq t < \infty, A \in \mathcal{B}_0
   (\mathbb{R}) \} . 
\end{equation*}
\begin{definition}
  \label{tx-Gwhitenoise}A random field $\{ \mathbf{W} ([s, t) \times
  A)\}_{([s, t) \times A) \in \Gamma}$ on a sublinear expectation space
  $(\Omega, \mathcal{H}, \hat{\mathbb{E}})$ is called a $1$-dimensional
  {space-time $G$-white noise} if it satisfies the following conditions:
  \begin{itemize}
    \item[(i)] For each fixed $[s, t)$, the spatial random field $\{ \mathbf{W} ([s,
    t) \times A)\}_{A \in \mathcal{B}_0 (\mathbb{R})}$ is a $1$-dimensional
    spatial $G$-white noise and has the same family of finite-dimensional
    distributions as $(\sqrt{t - s} \mathbb{W}_A)_{A \in \mathcal{B}_0
    (\mathbb{R})}$, where $(\mathbb{W}_A)_{A \in \mathcal{B}_0 (\mathbb{R})}$ is a
  $1$-dimensional spatial $G$-white noise with parameters $\underline{\sigma}^2$ and $\overline{\sigma}^2$;
    
    \item[(ii)] For every $0\leq r \leq s \leq t$ and $A \in \mathcal{B}_0
    (\mathbb{R})$, $\mathbf{W} ([r, t) \times A)= \mathbf{W} ([r, s) \times A) + \mathbf{W} ([s, t) \times
    A) $;
    
    \item[(iii)] For every $0\leq s_i \leq t_i \leq s\leq t$ and $A, A_i \in \mathcal{B}_0 (\mathbb{R})$,  $i = 1,
    \cdots, n$, $n\in\mathbb{N}$, $\mathbf{W} ([s, t) \times A)$ is independent of $(\mathbf{W}
    ([s_1, t_1) \times A_1), \cdots, \mathbf{W} ([s_n, t_n) \times A_n))$.
  \end{itemize}
  
\end{definition}

We work on the canonical realization of the space-time $G$-white noise constructed in {\cite{JP20,JP25}}. For completeness, we briefly recall the corresponding topological structure.

Set 
\begin{align*}
  \Gamma_1 = &\{ [s, t) \times (x,y): \forall0 \leq s \leq t < \infty, -\infty<x\leq y<\infty\}.
\end{align*}
Let $C_0(\Gamma_1;\mathbb{R})$ denote the space of continuous real-valued functions on $\Gamma_1$ vanishing on degenerate rectangles. Define
\begin{equation*} 
\Omega = \{ \omega \in C_0(\Gamma_1;\mathbb{R}): \omega (\gamma_1\cup \gamma_2) =
   \omega (\gamma_1) + \omega (\gamma_2),\text{ for } \gamma_1, \gamma_2\in\Gamma_1 \text{ with } \gamma_1\cup\gamma_2\in \Gamma_1 \text{ and } \gamma_1\cap \gamma_2=\emptyset\}
\end{equation*}
  endowed with the distance
  \begin{equation*}
   d(\omega,\tilde{\omega})=\sum_{k=1}^\infty2^{-k}\left[\left(\max_{\substack{0\leq s\leq t\leq k\\ -k\leq x\leq y\leq k}}|\omega([s,t)\times(x,y))-\tilde{\omega}([s,t)\times(x,y))|\right)\wedge1\right],\forall\omega,\tilde{\omega}\in\Omega.
   \end{equation*}

For each $\omega \in \Omega$, the canonical process
$(\mathbf{W}_{\gamma})_{\gamma \in \Gamma_1}$ is given by
$\mathbf{W}_\gamma(\omega) = \omega (\gamma)$, for $\gamma\in\Gamma_1$. 
For  $T \geq 0$, set $\mathcal{F}_T = \sigma \{ \mathbf{W} ([s, t) \times
(x,y)): 0 \leq s \leq t \leq T, -\infty<x\leq y<\infty\}$,
$\mathcal{F}= \bigvee_{T \geq 0} \mathcal{F}_T$, and
\begin{align*}
  L_{ip} (\mathcal{F}_T) := & \{\varphi (\mathbf{W} ([s_1, t_1) \times (x_1,y_1)),
  \cdots, \mathbf{W} ([s_n, t_n) \times (x_n,y_n))) : \text{for any } n \in
  \mathbb{N}, \\
  &\hspace{0.17em} \hspace{0.17em} \hspace{0.17em} \varphi \in C_{l.Lip}(\mathbb{R}^n), 0\leq s_i \leq t_i \leq T,-\infty<x_i\leq y_i<\infty  \text{ for } i = 1, \cdots, n  \},\\
  L_{ip} (\mathcal{F}) := & \cup_{n = 1}^{\infty} L_{ip} (\mathcal{F}_n) .
\end{align*}

In \cite{JP20}, we construct a sublinear expectation $\hat{\mathbb{E}}$ and a conditional expectation $\hat{\mathbb{E}}[\cdot|\mathcal{F}_t]$ on
$L_{ip} (\mathcal{F})$, under which the canonical process $(\mathbf{W}_\gamma)_{\gamma\in\Gamma_1}$ is a
$1$-dimensional space-time $G$-white noise. For each $p \geq 1$, $T \geq
0$, denote by $\mathbf{L}_G^p (\mathcal{F}_T)$ (resp., $\mathbf{L}_G^p
(\mathcal{F})$) the completion of $L_{ip} (\mathcal{F}_T)$ (resp., $L_{ip}
(\mathcal{F})$) under the norm $\|\cdot\|_p := (\hat{\mathbb{E}} [| \cdot
|^p])^{1 / p}$. Then the sublinear expectation and conditional expectation can be extended to the complete space $\mathbf{L}_G^p(\mathcal{F})$.

The canonical field can be extended from spatial intervals to every $A\in\mathcal{B}_0(\mathbb{R})$ by approximation in Lebesgue measure. In particular, if $\{A_n\}_{n=1}^\infty$ is a sequence of finite unions of bounded open intervals satisfying $\lim_{n\rightarrow \infty }\lambda_{A_n\triangle A}=0$, then  $\{\mathbf{W}([s,t)\times A_n)\}_{n=1}^\infty$ is a Cauchy sequence in $\mathbf{L}_G^p(\mathcal{F}_t)$ and denote the limit by $\mathbf{W}([s,t)\times A)\in \mathbf{L}_G^p(\mathcal{F}_t)$. Specifically, this limit is independent of the chosen approximating sequence. Thus, we construct a spatio-temporal $G$-white noise $\mathbf{W}=(\mathbf{W}_\gamma)_{\gamma\in\Gamma}$ indexed by $\Gamma$   under $\hat{\mathbb{E}}$. 

 For the sublinear expectation $\hat{\mathbb{E}}$, we have the following robust representation theorem which can be used to connect with the capacity theory.

\begin{theorem}[See \cite{DHP, HP, JP20}]\label{rep-thm}
There exists a weakly compact set $\mathcal{P}$ of probability measures on $(\Omega,\mathcal{B}(\Omega))$ such that
  \begin{equation*} 
  \hat{\mathbb{E}} [X] = \sup_{P \in \mathcal{P}} E_P [X],  \text{ for all }
     X \in \mathbf{L}_G^1(\mathcal{F}),
  \end{equation*}
 where $\mathcal{B}(\Omega)$ is the $\sigma$-algebra generated by all open sets.  $\mathcal{P}$ is called the  family of
  probability measures that represents $\hat{\mathbb{E}}$.
\end{theorem}

Let $\mathcal{P}$ be the weakly compact set that represents $\hat{\mathbb{E}}$.
For this set $\mathcal{P}$, the related capacity can be defined by
\begin{equation*} 
c (A) := \sup_{P \in \mathcal{P}} P (A), \text{ for }A \in \mathcal{B}(\Omega).
\end{equation*}
A set $A \in \mathcal{B}(\Omega)$ is called polar if $c (A) = 0$. A
property holds ``quasi-surely'' (q.s.) if it holds outside a polar set. In the
following, we do not distinguish two random variables $X$ and $Y$ if $X =Y$ q.s.

Let $ L^0(\Omega)$ denote the space of all $\mathcal{B}(\Omega)$-measurable real functions. We can define the related upper expectation for $X\in L^0(\Omega)$, making the following definition meaningful,
\begin{equation*}
\bar{\mathbb{E}}[X]:=\sup_{P\in\mathcal{P}}E_P[X], \text{ for } X\in L^0(\Omega).
\end{equation*}
Then $\hat{\mathbb{E}}=\bar{\mathbb{E}}$ on $L_{ip}({\mathcal{F}})$ and the completions of $L_{ip}({\mathcal{F}})$ under norms $(\hat{\mathbb{E}} [| \cdot
|^p])^{1 / p}$ and $(\bar{\mathbb{E}} [| \cdot
|^p])^{1 / p}$ are the same.

\begin{definition}
A random variable $X$ on $\Omega$ is called quasi-continuous if
for each $ \epsilon> 0$, there exists an open set $O$ with $c(O) <\epsilon$ such that $X|_{O^c}$ is continuous.

Furthermore, we say  $X$ has a quasi-continuous version if there exists a quasi-continuous
random variable $Y$ such that $ X = Y$ q.s.
\end{definition}

By Theorem 25 in Denis et al. \cite{DHP} and Proposition 4.4 in Hu and Peng \cite{HP}, we have the following pathwise characterization of $\mathbf{L}_G^p(\mathcal{F})$,
\begin{equation*}
\mathbf{L}_G^p(\mathcal{F})=\left\{X\in L^0(\Omega): X \text{ has a quasi-continuous version and }\lim_{N\rightarrow\infty}\bar{\mathbb{E}}[|X|^p\mathbf{1}_{\{|X|>N\}}]=0 \right\}.
\end{equation*}

In particular, if $X= Y$  in $\mathbf{L}_G^p(\mathcal{F})$, then $X = Y$ q.s. For $ X\in \mathbf{L}_G^p(\mathcal{F})$ and $\{X_n\}_{n=1}^\infty\subset\mathbf{L}_G^p(\mathcal{F})$, we say that $\{X_n\}_{n=1}^\infty$ converges to $X$ in $\mathbf{L}_G^p(\mathcal{F})$, denoted by $X=\mathbb{L}^p-\lim_{n\rightarrow\infty}X_n$, if 
\begin{equation*}
\lim_{n\rightarrow\infty}\|X_n-X\|_{\mathbf{L}_G^p}=\lim_{n\rightarrow\infty}\hat{\mathbb{E}}[|X_n-X|^p]^{1/p} = 0.
\end{equation*}

\subsection{Stochastic integrals with respect to space-time $G$-white noises}
For each given $p\geq1$, $T>0$,
let $\mathbf{M} ^{p, 0} ([0, T] \times \mathbb{R})$ be the collection of
simple random fields of the form:
\begin{equation}
  \eta (t, x ; \omega) := \sum_{i = 0}^{n - 1} \sum_{j = 1}^m X_{ij} (\omega) 
    \textbf{1}_{[t_i, t_{i + 1})} (t)\textbf{1}_{A_j} (x), \label{simple-func}
\end{equation}
where $X_{ij} \in \mathbf{L}_G^p (\mathcal{F}_{t_i})$, $i = 0, \cdots, n -
1$, $j = 1, \cdots, m$, $0 = t_0 < t_1 < \cdots < t_n = T$, and $\{A_j \}_{j =
1}^m \subset \mathcal{B}_0 (\mathbb{R})$ is a mutually disjoint sequence. The Bochner integral and stochastic integral for $\eta\in\mathbf{M} ^{p, 0} ([0, T] \times \mathbb{R})$ can be defined as follows:
\begin{align}
\int_0^T\int_{\mathbb{R}}\eta (t, x)dxdt&:=\sum_{i = 0}^{n - 1} \sum_{j = 1}^m X_{ij} (t_{i+1}-t_i)\lambda_{A_j},\label{bointeg}\\
\mathcal{I}_T(\eta):=\int_0^T\int_{\mathbb{R}}\eta (t, x)\mathbf{W} (dt,
    dx)&:=\sum_{i = 0}^{n - 1} \sum_{j = 1}^m X_{ij} \mathbf{W} ([t_i,t_{i+1})\times A_j).\label{stinteg}
\end{align}
These integrals (\ref{bointeg}) and (\ref{stinteg}) are linear mappings from $\mathbf{M} ^{p, 0} ([0, T] \times \mathbb{R})$ to $\mathbf{L}_G^p(\mathcal{F}_T)$. 

Set
\begin{align*} 
&\|\eta\|_{\mathbf{M}^p} := \left( \hat{\mathbb{E}} \left[ \int_0^T
   \int_{\mathbb{R}} |\eta (t, x) |^p dxdt \right] \right)^{\frac{1}{p}},\\ 
   &\|\eta\|_{\mathbf{H}^p} := \left( \hat{\mathbb{E}} \left[ \left(\int_0^T
   \int_{\mathbb{R}} |\eta (t, x) |^2 dxdt\right)^{\frac{p}{2}} \right] \right)^{\frac{1}{p}},\quad p\geq2,\\
   &\|\eta\|_{\mathbf{S}^p} := \sup_{t \in [0, T]} \sup_{x \in \mathbb{R}}
   (\hat{\mathbb{E}} [| \eta (t, x) |^p])^{\frac{1}{p}}, 
\end{align*}
and denote by $\mathbf{M}_G^p ([0, T] \times \mathbb{R})$, $\mathbf{H}_G^p ([0, T] \times \mathbb{R})$ and  $\mathbf{S}_G^p
([0, T] \times \mathbb{R})$ the completions of $\mathbf{M} ^{p, 0} ([0, T]
\times \mathbb{R})$ under the norm $\| \cdot \|_{\mathbf{M}^p}$, $\| \cdot \|_{\mathbf{H}^p}$ and $\|
\cdot \|_{\mathbf{S}^p}$, respectively. Then $\mathbf{H}_G^p([0, T] \times \mathbb{R})\subset\mathbf{H}_G^2([0, T] \times \mathbb{R})=\mathbf{M}_G^2([0, T] \times \mathbb{R})$ for $p\geq2$, and the integral (\ref{bointeg}) (resp., (\ref{stinteg})) can be continuously extended to the complete space $\mathbf{M}_G^1 ([0, T] \times \mathbb{R})$ (resp., $\mathbf{M}_G^2([0, T] \times \mathbb{R})$).  

Specifically, we mention that for the $\mathbf S^p$-norm, we adopt the convention that a simple
random field of the form \eqref{simple-func} takes at $t=T$ the same
value as on the last interval $[t_{n-1},T)$. This convention is
irrelevant for the $\mathbf M^p$-norm and $\mathbf H^p$-norm and for the
deterministic and stochastic integrals, since the single time point
$t=T$ has zero Lebesgue measure.

For the  stochastic integrals, we have the following important and useful properties.

\begin{lemma}[See Ji and Peng \cite{JP20,JP25}]\label{nle-1}
  For each $0 \leq s \leq r \leq T$, $\xi\in\mathbf{L}_G^2(\mathcal{F}_s)$, $\eta \in
  \mathbf{M}^2_G ([0, T] \times \mathbb{R})$, and $A, \bar{A} \in
  \mathcal{B}_0 (\mathbb{R})$ such that $A \cap \bar{A} = \emptyset$, we have
  \begin{itemize}
    \item[(i)] $\hat{\mathbb{E}} \left[ \xi\mathbf{W} ([s,r)\times A) \right] = 0$; 

    \item[(ii)] $\hat{\mathbb{E}} \left[ \pm \xi\mathbf{W} ([s,r)\times A)\mathbf{W} ([s,r)\times \bar{A}) \right] = 0$;
    
    \item[(iii)]  $\hat{\mathbb{E}} \left[ \pm \xi \left( \int_s^r \int_{\mathbb{R}} \eta (t, x) \mathbf{W} (dt, dx)
    \right) \right] = 0$;
    \item[(iv)] $\hat{\mathbb{E}} \left[ \pm \left( \int_0^s \int_{A} \eta (t, x)
    \mathbf{W} (dt, dx) \right) \left( \int_0^s \int_{\bar{A}} \eta (t, x) \mathbf{W}
    (dt, dx) \right) \right] = 0$;
    \item[(v)] $\hat{\mathbb{E}} \left[ \left|\int_0^T \int_{\mathbb{R}} \eta (t, x) \mathbf{W} (dt,
    dx) \right|^2\right] \leq \overline{\sigma}^2\hat{\mathbb{E}} \left[ \int_0^T \int_{\mathbb{R}} |\eta (t, x)|^2 dxdt \right]$.
  \end{itemize}
\end{lemma}

Denote by $L^2  ([0, T] \times \mathbb{R})$ the linear space of square-integrable functions on $[0, T] \times \mathbb{R}$, that is,
\begin{equation*}
L^2  ([0, T] \times \mathbb{R})=\left\{h:[0, T] \times \mathbb{R}\rightarrow\mathbb{R} \text{ s.t. } \int_0^T\int_{\mathbb{R}}|h(t,x)|^2dxdt< \infty\right\}.
\end{equation*}
For each $h (t, x) \in
L^2  ([0, T] \times \mathbb{R})$ and $\eta \in\mathbf{S}_G^2([0,T]\times \mathbb{R})$,   we can define the following  stochastic  
integral:
\begin{align*} 
\int_0^T \int_{\mathbb{R}} h (t, x) \eta (t, x) \mathbf{W} (dt, dx) &:=
   \mathbb{L}^2 - \lim_{n \rightarrow \infty} \int_0^T \int_{\mathbb{R}} h
   (t, x) \eta_n  (t, x) \mathbf{W} (dt, dx)\\
   &=\mathbb{L}^2 - \lim_{n \rightarrow \infty}\sum_{i = 0}^{m_n - 1} \sum_{j = 1}^{k_n} X_{ij}^{(n)}  \int_{t^{(n)}_i}^{t^{(n)}_{i + 1}}
  \int_{A^{(n)}_j} h(t, x)\mathbf{W} (dt, dx) , 
\end{align*}
where $\eta_n(t,x)=\sum_{i = 0}^{m_n - 1}\sum_{j = 1}^{k_n} X_{ij}^{(n)}   \textbf{1}_{[t^{(n)}_i, t^{(n)}_{i + 1})} (t)  \textbf{1}_{A^{(n)}_j} (x) \in
\mathbf{M}^{2, 0} ([0, T] \times \mathbb{R})$ satisfies that $\eta_n$ converges to $\eta$ in $\mathbf{S}_G^2([0,T]\times \mathbb{R})$ as $n\rightarrow\infty$.

\begin{lemma}[See Ji and Peng \cite{JP25}]
  \label{Ito ineq}For each $h \in L^2  ([0, T] \times \mathbb{R}), \eta
   \in \mathbf{S}_G^2([0,T]\times \mathbb{R})$, we have
  \begin{align}
    \label{nneq1} &\hat{\mathbb{E}} \left[  \int_0^T \int_{\mathbb{R}}
    h (t, x) \eta (t, x) \mathbf{W} (dt, dx)  \right] =0,\\
    &\label{neq-1} \hat{\mathbb{E}} \left[ \left| \int_0^T \int_{\mathbb{R}}
    h (t, x) \eta (t, x) \mathbf{W} (dt, dx) \right|^2 \right] \leq
    \overline{\sigma}^2 || \eta ||_{\mathbf{S}^2}^2 || h ||_{L^2}^2 .
  \end{align}
\end{lemma}

{Similarly,  for any given compact set $K\in\mathcal{B}_0(\mathbb{R})$, $h\in L^2([0,T]\times K)$ and $\eta\in\mathbf{S}_G^2([0,T]\times K)$, we can define the related Bochner integral from $\mathbf{S}_G^2([0,T]\times K)$ to $\mathbf{L}_G^2 (\mathcal{F}_T)$ as follows:
 \begin{align} \int_0^T \int_{K} h (t, x) \eta (t, x)  dxdt&:= \mathbb{L}^2 - \lim_{n \rightarrow \infty} \int_0^T \int_{K} h
   (t, x) \eta_n  (t, x) dxdt \label{b-int1}\\
   &=\mathbb{L}^2 - \lim_{n \rightarrow \infty}\sum_{i = 0}^{m_n - 1} \sum_{j = 1}^{k_n} X_{ij}^{(n)}  \int_{t^{(n)}_i}^{t^{(n)}_{i + 1}}
  \int_{A^{(n)}_j} h(t, x)dxdt,\label{b-int2}
 \end{align}
where $\eta_n(t,x)=\sum_{i = 0}^{m_n - 1}\sum_{j = 1}^{k_n} X_{ij}^{(n)}   \textbf{1}_{[t^{(n)}_i, t^{(n)}_{i + 1})} (t)  \textbf{1}_{A^{(n)}_j} (x) \in
\mathbf{M}^{2, 0} ([0, T] \times K)$ satisfies that $\eta_n$ converges to $\eta$ in $\mathbf{S}_G^2([0,T]\times K)$ as $n\rightarrow\infty$.}

\section{Quadratic variation processes and a BDG inequality}
We next develop the quadratic variation theory associated with the space-time $G$-white noise and establish a BDG inequality.

\subsection{Quadratic  and mutual variation processes}
For the space-time $G$-white noise $\{\mathbf{W}([s,t)\times A)\}_{([s,t)\times A)\in\Gamma}$, we now fix the initial time $s=0$, and set 
\begin{equation*}
  \mathbf{W}^A_t :=\mathbf{W}([0,t)\times A), \text{ for } t\geq0, A\in\mathcal{B}_0(\mathbb{R}).
\end{equation*} 
Then for each fixed $A\in\mathcal{B}_0(\mathbb{R})$, by Definition \ref{tx-Gwhitenoise}, it is easy to check that $\mathbf{W}^A=(\mathbf{W}^A_t)_{t\geq0}$ is a $G$-Brownian motion with upper variance $\hat{\mathbb{E}}[|\mathbf{W}^A_t|^2]=\overline{\sigma}^2\lambda_At$ and lower variance $-\hat{\mathbb{E}}[-|\mathbf{W}^A_t|^2]=\underline{\sigma}^2\lambda_At$. Hence, following Peng \cite{P2019}, the quadratic variation process of $\mathbf{W}^A$ can be defined by 
\begin{align}
\langle \mathbf{W}^A\rangle_t&:=\mathbb{L}^2-\lim_{\mu(\pi_t^N)\rightarrow0}\sum_{j=0}^{N-1}(\mathbf{W}^A_{t_{j+1}^N}-\mathbf{W}^A_{t_{j}^N})^2  \nonumber\\
&=(\mathbf{W}^A_t)^2-2\int_0^t\mathbf{W}^A_sd\mathbf{W}^A_s,
\end{align}
where $\pi_t^N=\{t_{0}^N,t_{1}^N,\cdots,t_{N}^N\}$, $N=1,2,\cdots$, denote a sequence of partitions of  $[0,t]$, and $\mu(\pi_t^N):=\max\{|t_{j+1}^N-t^N_j|,j=0, 1,\cdots,N-1\}$. 
By Sections 3.4 and 3.5 in Peng \cite{P2019}, $\langle \mathbf{W}^A\rangle_t$ has the following properties.

\begin{proposition}[See Peng \cite{P2019}]\label{prop:quad-var}
For each $A\in\mathcal{B}_0(\mathbb{R})$, the process
$\langle\mathbf W^A\rangle
=
\bigl(
\langle\mathbf W^A\rangle_t
\bigr)_{t\ge0}$
has quasi-surely continuous and nondecreasing paths, with
$\langle\mathbf W^A\rangle_0=0$.
Moreover, for every $0\le s<t<\infty$, we have
  \begin{itemize}
    \item[(i)]  $\langle \mathbf{W}^A\rangle_t$ is maximally distributed and  $\langle \mathbf{W}^A\rangle_t\sim\mathcal{M}([\underline{\sigma}^2\lambda_At,\overline{\sigma}^2\lambda_At])$;
    \item[(ii)] $\langle \mathbf{W}^A\rangle_t-\langle \mathbf{W}^A\rangle_s\overset{d}{=}\langle \mathbf{W}^A\rangle_{t-s}$;
    \item[(iii)] $\langle \mathbf{W}^A\rangle_s\leq\langle \mathbf{W}^A\rangle_t$;
    \item[(iv)] $\underline{\sigma}^2\lambda_A(t-s)\leq \langle \mathbf{W}^A\rangle_t-\langle \mathbf{W}^A\rangle_s\leq\overline{\sigma}^2\lambda_A(t-s)$;
    \item[(v)] For $0\leq r<u\leq s<t$, $\langle \mathbf{W}^A\rangle_t-\langle \mathbf{W}^A\rangle_s$ is independent of $\langle \mathbf{W}^A\rangle_u-\langle \mathbf{W}^A\rangle_r$.
  \end{itemize}
\end{proposition}

For $A,B\in\mathcal B_0(\mathbb R)$, we next consider the mutual
quadratic variation between the spatial projections $
\mathbf W^A$ and
$\mathbf W^B$.
Let
\[
\pi=\{0=t_0<t_1<\cdots<t_N=t\}
\]
be any deterministic partition of $[0,t]$, and define the cross-variation sum
\begin{equation}
\label{eq:cross-sum}
C_t^{\pi}(A,B)
:=
\sum_{k=0}^{N-1}(\mathbf W^A_{t_{k+1}}-\mathbf W^A_{t_k})(\mathbf W^B_{t_{k+1}}-\mathbf W^B_{t_k}).
\end{equation}

\begin{lemma}
\label{lem:cross-zero}
Let $A,B\in\mathcal B_0(\mathbb R)$ satisfy
$A\cap B=\emptyset$. Then, for every $t\geq0$,
\begin{equation}
\label{eq:disjoint-cross-qv}
\mathbb{L}^2-\lim_{\mu(\pi)\to0}
C_t^\pi(A,B)
=
0.
\end{equation}
\end{lemma}

\begin{proof}Let
$\pi=\{0=t_0<t_1<\cdots<t_N=t\}$
be any deterministic partition of $[0,t]$. For $0\leq k<\ell\leq N-1$,
\[
(\mathbf W^A_{t_{k+1}}-\mathbf W^A_{t_k})(\mathbf W^B_{t_{k+1}}-\mathbf W^B_{t_k})\in\mathbf L_G^2(\mathcal F_{t_\ell}),
\]
and Lemma \ref{nle-1}
implies
\begin{equation}
\label{eq:cross-orthogonality}
\hat{\mathbb E}[\pm(\mathbf W^A_{t_{k+1}}-\mathbf W^A_{t_k})(\mathbf W^B_{t_{k+1}}-\mathbf W^B_{t_k})(\mathbf W^A_{t_{\ell+1}}-\mathbf W^A_{t_\ell})(\mathbf W^B_{t_{\ell+1}}-\mathbf W^B_{t_\ell})]
=
0.
\end{equation}
Recall that if a random variable $Y$ satisfies
$\hat{\mathbb E}[Y]
=
\hat{\mathbb E}[-Y]
=
0,$
then
$\hat{\mathbb E}[X+Y]
=
\hat{\mathbb E}[X]$
for every integrable random variable $X$. Hence all the cross terms in
the expansion of $\bigl(C_t^\pi(A,B)\bigr)^2$ can be removed, and
\begin{equation}
\label{eq:cross-second-moment}
\hat{\mathbb E}
\left[
\left|C_t^\pi(A,B)\right|^2
\right]
=
\hat{\mathbb E}
\left[
\sum_{k=0}^{N-1}|(\mathbf W^A_{t_{k+1}}-\mathbf W^A_{t_k})(\mathbf W^B_{t_{k+1}}-\mathbf W^B_{t_k})|^2
\right]
\leq
\sum_{k=0}^{N-1}
\hat{\mathbb E}[(\mathbf W^A_{t_{k+1}}-\mathbf W^A_{t_k})^2(\mathbf W^B_{t_{k+1}}-\mathbf W^B_{t_k})^2].
\end{equation}
By   H\"older's inequality,
\[
\hat{\mathbb E}
\left[
(\mathbf W^A_{t_{k+1}}-\mathbf W^A_{t_k})^2(\mathbf W^B_{t_{k+1}}-\mathbf W^B_{t_k})^2
\right]
\leq
\hat{\mathbb E}
\left[
|\mathbf W^A_{t_{k+1}}-\mathbf W^A_{t_k}|^4
\right]^{1/2}
\hat{\mathbb E}
\left[
|\mathbf W^B_{t_{k+1}}-\mathbf W^B_{t_k}|^4
\right]^{1/2}.
\]
Since the increments are one-dimensional $G$-normally distributed, we have
\[
\hat{\mathbb E}
\left[
|\mathbf W^A_{t_{k+1}}-\mathbf W^A_{t_k}|^4
\right]
\leq
3\overline{\sigma}^4
\lambda_A^2
(t_{k+1}-t_k)^2,
\]
and similarly for $B$. Consequently,
\[
\hat{\mathbb E}[(\mathbf W^A_{t_{k+1}}-\mathbf W^A_{t_k})^2(\mathbf W^B_{t_{k+1}}-\mathbf W^B_{t_k})^2]
\leq
3\overline{\sigma}^4
\lambda_A\lambda_B
(t_{k+1}-t_k)^2.
\]
It follows from \eqref{eq:cross-second-moment} that
\[
\begin{aligned}
\hat{\mathbb E}
\left[
|C_t^\pi(A,B)|^2
\right]
&\leq
3\overline{\sigma}^4
\lambda_A\lambda_B
\sum_{k=0}^{N-1}(t_{k+1}-t_k)^2
\\
&\leq
3\overline{\sigma}^4
\lambda_A\lambda_Bt\mu(\pi).
\end{aligned}
\]
Letting $\mu(\pi)\to0$ proves
\eqref{eq:disjoint-cross-qv}.
\end{proof}

\begin{corollary}
\label{cor:qv-additivity}
If $A,B\in\mathcal{B}_0(\mathbb{R})$ are disjoint, then, for every $t\geq0$,
\begin{equation}\label{eq:qv-additivity}
\langle\mathbf W^{A\cup B}\rangle_t
=
\langle\mathbf W^A\rangle_t
+
\langle\mathbf W^B\rangle_t.
\end{equation}
\end{corollary}

\begin{proof} If $A\cap B=\emptyset$, then
\[
\mathbf W^{A\cup B}_t
=
\mathbf W^A_t+\mathbf W^B_t, \text{ for } t\geq0.
\]
For every partition $\pi$,
\[
\sum_k(\mathbf W^{A\cup B}_{t_{k+1}}-\mathbf W^{A\cup B}_{t_k})^2
=
\sum_k(\mathbf W^A_{t_{k+1}}-\mathbf W^A_{t_k})^2
+
\sum_k(\mathbf W^B_{t_{k+1}}-\mathbf W^B_{t_k})^2
+
2C_t^\pi(A,B).
\]
The first two sums converge in $\mathbf L_G^2$ to the corresponding
$G$-quadratic variations, while the last term converges to zero by
Lemma~\ref{lem:cross-zero}.  Passing to the $\mathbf L_G^2$ limit proves
\eqref{eq:qv-additivity}.
\end{proof}


\subsection{Stochastic integral with respect to the quadratic variation processes}

For every $\eta=\sum_{i=0}^{n-1}\sum_{j=1}^mX_{ij}(\omega)\mathbf{1}_{[t_i,t_{i+1})}(t)\mathbf{1}_{A_j}(x)\in \mathbf{M}^{1,0}([0, T ] \times \mathbb{R})$, the stochastic integral of $\eta$ with respect to $\langle \mathbf{W}\rangle$ can be defined by 
\begin{equation}\label{eq:int-quavar}
\mathcal Q_T(\eta)=\int_{[0,T]\times\mathbb{R}}\eta(t,x)d\langle \mathbf{W}\rangle:=\sum_{i=0}^{n-1}\sum_{j=1}^mX_{ij}\left(\langle\mathbf{W}^{A_j}\rangle_{t_{i+1}}-\langle\mathbf{W}^{A_j}\rangle_{t_{i}}\right),
\end{equation} 
For $0\le t\le T$ we write
\begin{equation}
\label{eq:qv-integral-time}
\mathcal Q_t(\eta)
:=
\mathcal Q_T
\bigl(
\eta\,\mathbf1_{[0,t)}
\bigr).
\end{equation}
Corollary \ref{cor:qv-additivity} implies that this integral  $\mathcal Q_T$ is independent of the particular disjoint spatial representation of $\eta$. Moreover, the integral  $\mathcal Q_T$ is a linear mapping from $\mathbf{M}^{1,0}([0, T ] \times \mathbb{R})$  to $\mathbf{L}^1_G(\mathcal{F}_T)$,  and can be continuously extended to the complete space $\mathbf{M}_G^1([0,T]\times\mathbb{R})$.


\begin{lemma}\label{le:Ito-ineq}
For every $\eta\in \mathbf{M}^{1,0}([0, T ] \times \mathbb{R})$, we have 
\begin{equation}\label{eq:Ito-ineq}
  \hat{\mathbb{E}}\left[\left|\int_{[0,T]\times\mathbb{R}}\eta(t,x)d\langle \mathbf{W}\rangle\right|\right]\leq\overline{\sigma}^2\hat{\mathbb{E}}\left[\int_{\mathbb{R}}\int_0^T\left|\eta(t,x)\right|dtdx\right].
\end{equation}
\end{lemma}
\begin{proof}
  For every $\eta=\sum_{i=0}^{n-1}\sum_{j=1}^mX_{ij}(\omega)\mathbf{1}_{[t_i,t_{i+1})}(t)\mathbf{1}_{A_j}(x)\in \mathbf{M}^{1,0}([0, T ] \times \mathbb{R})$,
  \begin{align*}
  \hat{\mathbb{E}}\left[\left|\int_{[0,T]\times\mathbb{R}}\eta(t,x)d\langle \mathbf{W}\rangle\right|\right]&=\hat{\mathbb{E}}\left[\left|\sum_{i=0}^{n-1}\sum_{j=1}^mX_{ij}\left(\langle\mathbf{W}^{A_j}\rangle_{t_{i+1}}-\langle\mathbf{W}^{A_j}\rangle_{t_{i}}\right)\right|\right]\\
  &\leq \hat{\mathbb{E}}\left[\sum_{i=0}^{n-1}\sum_{j=1}^m\left|X_{ij}\right|\left|\langle\mathbf{W}^{A_j}\rangle_{t_{i+1}}-\langle\mathbf{W}^{A_j}\rangle_{t_{i}}\right|\right]\\
  &\leq \hat{\mathbb{E}}\left[\xi\right]+\overline{\sigma}^2\hat{\mathbb{E}}\left[\sum_{i=0}^{n-1}\sum_{j=1}^m\left|X_{ij}\right|(t_{i+1}-t_{i})\lambda_{A_j}\right],
  \end{align*}
where
\begin{align*}
  \hat{\mathbb{E}}\left[\xi\right]&=\hat{\mathbb{E}}\left[\sum_{i=0}^{n-1}\sum_{j=1}^m\left|X_{ij}\right|\left(\left|\langle\mathbf{W}^{A_j}\rangle_{t_{i+1}}-\langle\mathbf{W}^{A_j}\rangle_{t_{i}}\right|-\overline{\sigma}^2(t_{i+1}-t_{i})\lambda_{A_j}\right)\right]\\
  &\leq\sum_{i=0}^{n-1}\sum_{j=1}^m\hat{\mathbb{E}}\left[\left|X_{ij}\right|\left(\left|\langle\mathbf{W}^{A_j}\rangle_{t_{i+1}}-\langle\mathbf{W}^{A_j}\rangle_{t_{i}}\right|-\overline{\sigma}^2(t_{i+1}-t_{i})\lambda_{A_j}\right)\right]\\
  &\leq\sum_{i=0}^{n-1}\sum_{j=1}^m\hat{\mathbb{E}}\left[\left|X_{ij}\right|\left(\hat{\mathbb{E}}\left[\left|\langle\mathbf{W}^{A_j}\rangle_{t_{i+1}}-\langle\mathbf{W}^{A_j}\rangle_{t_{i}}\right|\Big|\mathcal{F}_{t_i}\right]-\overline{\sigma}^2(t_{i+1}-t_{i})\lambda_{A_j}\right)\right]\\
  &=\sum_{i=0}^{n-1}\sum_{j=1}^m\hat{\mathbb{E}}\left[\left|X_{ij}\right|\left(\hat{\mathbb{E}}\left[\left|\langle\mathbf{W}^{A_j}\rangle_{t_{i+1}}-\langle\mathbf{W}^{A_j}\rangle_{t_{i}}\right|\right]-\overline{\sigma}^2(t_{i+1}-t_{i})\lambda_{A_j}\right)\right]\\
  &=0.
  \end{align*}
  Hence,
  \begin{equation*}
  \hat{\mathbb{E}}\left[\left|\int_{[0,T]\times\mathbb{R}}\eta(t,x)d\langle \mathbf{W}\rangle\right|\right]\leq \overline{\sigma}^2\hat{\mathbb{E}}\left[\sum_{i=0}^{n-1}\sum_{j=1}^m\left|X_{ij}\right|(t_{i+1}-t_{i})\lambda_{A_j}\right]=\overline{\sigma}^2\hat{\mathbb{E}}\left[\int_{\mathbb{R}}\int_0^T\left|\eta(t,x)\right|dtdx\right],
  \end{equation*}
which completes the proof.
\end{proof}

Thus, the stochastic integral (\ref{eq:int-quavar}) can be continuously extended to the completion space $\mathbf{M}_G^1([0,T]\times\mathbb{R})$ and the  inequality (\ref{eq:Ito-ineq}) still holds in $\mathbf{M}_G^1([0,T]\times\mathbb{R})$.

Similar to Proposition 3.4.5 in Peng \cite{P2019}, we have the following $G$-It\^o isometry formula. 

\begin{lemma}\label{le:Ito-iso}
For $\eta\in\mathbf{M}_G^2([0,T]\times\mathbb{R})$,
\begin{equation}\label{eq:Ito-iso}
  \hat{\mathbb{E}}\left[\left|\int_{\mathbb{R}}\int_0^T\eta(t,x)\mathbf{W}(dt,dx)\right|^2\right]=\hat{\mathbb{E}}\left[\int_{[0,T]\times\mathbb{R}}\left|\eta(t,x)\right|^2d\langle \mathbf{W}\rangle\right].
\end{equation}
\end{lemma}
To prove the above It\^o isometry, we need the following lemma, which is a natural generalization of Proposition 3.4.4 in Peng \cite{P2019}.
\begin{lemma}\label{le:w2-<W>}
For $X\in \mathbf{L}^2_G(\mathcal{F})$ and $\xi_j\in \mathbf{L}^2_G(\mathcal{F}_s)$, $1\leq j\leq m$,
\begin{equation}
  \hat{\mathbb{E}}\left[X+\sum_{j=1}^m\xi_j^2|\mathbf{W}([s,t)\times A_j)|^2\right]=\hat{\mathbb{E}}\left[X+\sum_{j=1}^m\xi_j^2\left(\langle \mathbf{W}^{A_j}\rangle_t-\langle \mathbf{W}^{A_j}\rangle_s\right)\right].
\end{equation}
\end{lemma}
\begin{proof}
\begin{align*}
\hat{\mathbb{E}}\left[X+\sum_{j=1}^m\xi_j^2|\mathbf{W}([s,t)\times A_j)|^2\right]&=\hat{\mathbb{E}}\left[X+\sum_{j=1}^m\xi_j^2\left| \mathbf{W}^{A_j}_t-\mathbf{W}^{A_j}_s\right|^2\right]\\
&=\hat{\mathbb{E}}\left[X+\sum_{j=1}^m\xi_j^2\left(\left|\mathbf{W}^{A_j}_t\right|^2+\left|\mathbf{W}^{A_j}_s\right|^2-2\mathbf{W}^{A_j}_t\mathbf{W}^{A_j}_s\right)\right]\\
&=\hat{\mathbb{E}}\left[X+\sum_{j=1}^m\xi_j^2\left(\left|\mathbf{W}^{A_j}_t\right|^2-\left|\mathbf{W}^{A_j}_s\right|^2-2\left(\mathbf{W}^{A_j}_t-\mathbf{W}^{A_j}_s\right)\mathbf{W}^{A_j}_s\right)\right]\\
&=\hat{\mathbb{E}}\left[X+\sum_{j=1}^m\xi_j^2\left(\left|\mathbf{W}^{A_j}_t\right|^2-\left|\mathbf{W}^{A_j}_s\right|^2\right)\right].
\end{align*}
Since 
\begin{equation*}
  \hat{\mathbb{E}}\left[\pm\xi_j^2\int_s^t\mathbf{W}^{A_j}_rd\mathbf{W}^{A_j}_r\right]=0,
\end{equation*}
we have
\begin{align*}
  &\hat{\mathbb{E}}\left[X+\sum_{j=1}^m\xi_j^2\left(\left|\mathbf{W}^{A_j}_t\right|^2-\left|\mathbf{W}^{A_j}_s\right|^2\right)\right]\\
  &=\hat{\mathbb{E}}\left[X+\sum_{j=1}^m\xi_j^2\left(\left|\mathbf{W}^{A_j}_t\right|^2-\left|\mathbf{W}^{A_j}_s\right|^2-2\int_s^t\mathbf{W}^{A_j}_rd\mathbf{W}^{A_j}_r\right)\right]\\
  &=\hat{\mathbb{E}}\left[X+\sum_{j=1}^m\xi_j^2\left(\langle \mathbf{W}^{A_j}\rangle_t-\langle \mathbf{W}^{A_j}\rangle_s\right)\right],
\end{align*}
which completes the proof.
\end{proof}

\begin{proof}[Proof of Lemma \ref{le:Ito-iso}] 
For each $\eta=\sum_{i=0}^{n-1}\sum_{j=1}^mX_{ij}(\omega)\mathbf{1}_{[t_i,t_{i+1})}(t)\mathbf{1}_{A_j}(x)\in \mathbf{M}^{2,0}([0, T ] \times \mathbb{R})$,
\begin{align*}
  \hat{\mathbb{E}}\left[\left|\int_{\mathbb{R}}\int_0^T\eta(t,x)\mathbf{W}(dt,dx)\right|^2\right]
  &=\hat{\mathbb{E}}\left[\left|\sum_{i=0}^{n-1}\sum_{j=1}^mX_{ij}\mathbf{W}([t_i,t_{i+1})\times A_j)\right|^2\right]\\
  &=\hat{\mathbb{E}}\left[\sum_{i=0}^{n-1}\sum_{j=1}^mX^2_{ij}\left|\mathbf{W}([t_i,t_{i+1})\times A_j)\right|^2\right]\\
  &=\hat{\mathbb{E}}\left[\sum_{i=0}^{n-1}\sum_{j=1}^mX^2_{ij}\left(\langle \mathbf{W}^{A_j}\rangle_{t_{i+1}}-\langle \mathbf{W}^{A_j}\rangle_{t_{i}}\right)\right]\\
  &=\hat{\mathbb{E}}\left[\int_{\mathbb{R}}\int_0^T\left|\eta(t,x)\right|^2d\langle \mathbf{W}\rangle\right],
\end{align*}
where the third equality applies Lemma \ref{le:w2-<W>}.

For every $\eta\in\mathbf{M}^{2}_G([0, T ] \times \mathbb{R})$, there exists a sequence of $\{\eta_n\}_{n=1}^\infty$ such that $\eta_n\to \eta$ in $\mathbf{M}^{2}_G([0, T ] \times \mathbb{R})$ as $n\to \infty$. Then 
\begin{align*}
\hat{\mathbb{E}}\left[\left|\int_{\mathbb{R}}\int_0^T\eta(t,x)\mathbf{W}(dt,dx)\right|^2\right]&=\lim_{n\to\infty}\hat{\mathbb{E}}\left[\left|\int_{\mathbb{R}}\int_0^T\eta_n(t,x)\mathbf{W}(dt,dx)\right|^2\right]\\
&=\lim_{n\to\infty}\hat{\mathbb{E}}\left[\int_{\mathbb{R}}\int_0^T\left|\eta_n(t,x)\right|^2d\langle \mathbf{W}\rangle\right]\\
&=\hat{\mathbb{E}}\left[\int_{\mathbb{R}}\int_0^T\left|\eta(t,x)\right|^2d\langle \mathbf{W}\rangle\right]. 
\end{align*} 
\end{proof}

\subsection{Identification under the representing measures and the BDG inequality}

Let $\mathcal{P}$ be the weakly compact family of probability measures that represents $\hat{\mathbb{E}}$ in
Theorem \ref{rep-thm}.  Recall that, for a nonnegative
$\mathcal{B}(\Omega)$-measurable random variable $X$,
\[
\bar{\mathbb{E}}[X]
:=
\sup_{P\in\mathcal{P}}E_P[X].
\]
On $\mathbf{L}_G^1(\mathcal F)$, $\bar{\mathbb E}$ coincides with
$\hat{\mathbb E}$.

We first identify, under each $P\in\mathcal P$, the classical quadratic
variations of the spatial projections of the noise.

\begin{lemma}\label{le:P-mart}
For every $A\in\mathcal B_0(\mathbb R)$ and every $P\in\mathcal P$,
$\mathbf W^A=(\mathbf W_t^A)_{t\geq 0}$ is a continuous
square-integrable $P$-martingale with respect to the usual
$P$-augmentation of the canonical filtration.
\end{lemma}

\begin{proof}
Let $0\le s<t$.  For each bounded
Lipschitz cylinder random variable $\xi\in L_{ip} (\mathcal{F}_s)$, Lemma~\ref{nle-1}
applied to $\xi$ and $-\xi$ yields
\[
\hat{\mathbb E}\!\left[
\pm \xi(\mathbf W_t^A-\mathbf W_s^A)
\right]
=0.
\]
Since $E_P[X]\le\hat{\mathbb E}[X]$ for every
$X\in\mathbf L_G^1(\mathcal F)$, it follows that
\[
E_P\!\left[\xi(\mathbf W_t^A-\mathbf W_s^A)\right]=0.
\]
A monotone class argument extends this identity to bounded random variables
measurable with respect to the $P$-completion of $\mathcal F_s$. Then $\mathbf W^A$ is a martingale
with respect to the usual $P$-augmentation.   Moreover,
\[
E_P[|\mathbf W_t^A|^2]
\le
\hat{\mathbb E}[|\mathbf W_t^A|^2]
=
\overline{\sigma}^2\lambda_A t<\infty.
\]
The continuity follows from the quasi-surely continuous version of the one-dimensional
$G$-Brownian motion $\mathbf W^A$.  Thus $\mathbf W^A$ is a continuous
square-integrable $P$-martingale.
\end{proof}

For  continuous $P$-martingales $M$ and $N$, denote its classical quadratic
variation by $[M]^P$ and the classical cross-variation of $M,N$ by
$[M,N]^P$.

\begin{proposition}\label{prop:P-qv-identification}
For every $P\in\mathcal P$ and $A\in\mathcal B_0(\mathbb R)$,
\[
[\mathbf W^A]^P_t
=
\langle\mathbf W^A\rangle_t,
\qquad
0\le t\le T,
\quad P\text{-a.s.}
\]
where the equality holds simultaneously for all $t$, i.e. the two
continuous processes are indistinguishable under $P$.  If
$A,B\in\mathcal B_0(\mathbb R)$ are disjoint, then
\[
[\mathbf W^A,\mathbf W^B]^P_t
=
0,
\qquad
0\le t\le T,
\quad P\text{-a.s.}
\]
\end{proposition}

\begin{proof}
Fix first $t\in[0,T]$ and let
$\pi_n=\{0=t_0^n<\cdots<t_{N_n}^n=t\}$ be deterministic partitions with
$\mu(\pi_n)\to0$.  By the definition of the $G$-quadratic variation,
\[
\sum_{k=0}^{N_n-1}
(\mathbf W^A_{t_{k+1}^n}-\mathbf W^A_{t_k^n})^2
\longrightarrow
\langle\mathbf W^A\rangle_t
\qquad\text{in }\mathbf L_G^2.
\]
Hence the same convergence holds in $P$-probability.  On the other hand,
by Lemma~\ref{le:P-mart} and the classical quadratic-variation theorem,
the same sums converge in $P$-probability to $[\mathbf W^A]^P_t$.
Uniqueness of limits in probability therefore gives
\[
[\mathbf W^A]^P_t=\langle\mathbf W^A\rangle_t,
\qquad P\text{-a.s.}
\]
for each fixed $t$.  Applying this equality simultaneously to all
$t\in\mathbb Q\cap[0,T]$ and then using continuity of both processes yields
the equality for all $t\in[0,T]$ outside one $P$-null set.

Now suppose that $A\cap B=\emptyset$.  By Lemma~\ref{lem:cross-zero},
for every fixed $t$,
\[
\sum_{k=0}^{N_n-1}
(\mathbf W^A_{t_{k+1}^n}-\mathbf W^A_{t_k^n})
(\mathbf W^B_{t_{k+1}^n}-\mathbf W^B_{t_k^n})
\longrightarrow0
\qquad\text{in }\mathbf L_G^2,
\]
and hence in $P$-probability.  Classical quadratic-variation theory gives
convergence of the same cross sums in $P$-probability to
$[\mathbf W^A,\mathbf W^B]^P_t$.  Thus the cross-variation is zero at every
fixed $t$, $P$-a.s.  A rational-time argument and the continuity of the
classical cross-variation process yield the asserted process level identity.
\end{proof}

We next compute the bracket of the stochastic integral for simple
integrands.  Let
\[
\eta(t,x)
=
\sum_{i=0}^{n-1}\sum_{j=1}^{m}
X_{ij}
\mathbf1_{[t_i,t_{i+1})}(t)
\mathbf1_{A_j}(x)
\in\mathbf M^{p,0}([0,T]\times\mathbb R),
\qquad p\ge2,
\]
where the $A_j$ are mutually disjoint, and set
\[
\mathcal{I}_t(\eta)
:=
\int_0^t\int_{\mathbb R}
\eta(s,x)\mathbf W(ds,dx).
\]
Define also
\begin{equation}\label{eq:A-simple}
\mathcal A_t(\eta)
:=
\int_{[0,t]\times\mathbb R}
|\eta(s,x)|^2d\langle\mathbf W\rangle
=
\sum_{i=0}^{n-1}\sum_{j=1}^{m}
X_{ij}^2
\left(
\langle\mathbf W^{A_j}\rangle_{t\wedge t_{i+1}}
-
\langle\mathbf W^{A_j}\rangle_{t\wedge t_i}
\right).
\end{equation}
For simple $\eta$, the process $\mathcal A(\eta)$ is quasi-surely continuous
and nondecreasing.
\begin{lemma}\label{le:simple-bracket}
For every $P\in\mathcal P$,
$\mathcal I(\eta)=(\mathcal I_t(\eta))_{0\le t\le T}$
is a continuous square-integrable $P$-martingale. Moreover,
up to $P$-indistinguishability,
\begin{equation}\label{eq:simple-bracket}
[\mathcal I(\eta)]^P_t
=
\mathcal A_t(\eta),
\qquad 0\le t\le T.
\end{equation}
\end{lemma}

\begin{proof}
For $1\le j\le m$, define
\[
\phi_j(s)
:=
\sum_{i=0}^{n-1}
X_{ij}\mathbf1_{(t_i,t_{i+1}]}(s).
\]
Since $X_{ij}$ is $\mathcal F_{t_i}$-measurable,
$\phi_j$ is a predictable step process under every
$P\in\mathcal P$.

We first identify the simple space-time stochastic integral with a
classical stochastic integral under $P$. By the definition of
$\mathcal I(\eta)$,
\[
\mathcal I_t(\eta)
=
\sum_{i=0}^{n-1}\sum_{j=1}^m
X_{ij}
\left(
\mathbf W_{t\wedge t_{i+1}}^{A_j}
-
\mathbf W_{t\wedge t_i}^{A_j}
\right).
\]
On the other hand, by the definition of the classical stochastic
integral for predictable step processes,
\[
\int_0^t\phi_j(s)\,d\mathbf W_s^{A_j}
=
\sum_{i=0}^{n-1}
X_{ij}
\left(
\mathbf W_{t\wedge t_{i+1}}^{A_j}
-
\mathbf W_{t\wedge t_i}^{A_j}
\right).
\]
Therefore,
\begin{equation}\label{eq:simple-P-representation}
\mathcal I_t(\eta)
=
\sum_{j=1}^m
\int_0^t\phi_j(s)\,d\mathbf W_s^{A_j},
\qquad 0\le t\le T,
\end{equation}
up to $P$-indistinguishability.
Notice that the original convention
$\mathbf1_{[t_i,t_{i+1})}$ for the simple random field and the
predictable convention $\mathbf1_{(t_i,t_{i+1}]}$ lead to the same
finite-increment expression above; hence no change in the definition of
the space-time stochastic integral is involved.

We next verify square integrability. By
Proposition~\ref{prop:P-qv-identification},
$[\mathbf W^{A_j}]^P
=
\langle\mathbf W^{A_j}\rangle$
up to $P$-indistinguishability. Hence
\begin{align*}
E_P\left[
\int_0^T|\phi_j(s)|^2
\,d[\mathbf W^{A_j}]_s^P
\right]
&=
E_P\left[
\sum_{i=0}^{n-1}
X_{ij}^2
\left(
\langle\mathbf W^{A_j}\rangle_{t_{i+1}}
-
\langle\mathbf W^{A_j}\rangle_{t_i}
\right)
\right].
\end{align*}
For each fixed pair $(t_i,t_{i+1})$,
Proposition~\ref{prop:quad-var} gives
\[
\langle\mathbf W^{A_j}\rangle_{t_{i+1}}
-
\langle\mathbf W^{A_j}\rangle_{t_i}
\le
\overline\sigma^2\lambda_{A_j}
(t_{i+1}-t_i)
\qquad\text{q.s.}
\]
Since only finitely many indices $i$ are involved, the corresponding
exceptional polar sets can be replaced by their finite union. Therefore,
\begin{align*}
E_P\left[
\int_0^T|\phi_j(s)|^2
\,d[\mathbf W^{A_j}]_s^P
\right]
&\le
\overline\sigma^2\lambda_{A_j}
\sum_{i=0}^{n-1}
(t_{i+1}-t_i)E_P[X_{ij}^2]
\\
&\le
\overline\sigma^2\lambda_{A_j}
\sum_{i=0}^{n-1}
(t_{i+1}-t_i)
\hat{\mathbb E}[X_{ij}^2]
<\infty.
\end{align*}
Hence $
\int_0^\cdot \phi_j(s)\,d\mathbf W_s^{A_j}$
is a continuous square-integrable $P$-martingale. Since the sum over
$j$ is finite, $\mathcal I(\eta)$ is also a continuous
square-integrable $P$-martingale.

As
the $A_j$ are mutually disjoint, Proposition~\ref{prop:P-qv-identification}
and the classical bracket formula yield
\begin{align*}
[\mathcal{I}(\eta)]^P_t
&=
\sum_{j=1}^{m}
\int_0^t|\phi_j(s)|^2\,d[\mathbf W^{A_j}]^P_s
\\
&=
\sum_{j=1}^{m}
\int_0^t|\phi_j(s)|^2\,d\langle\mathbf W^{A_j}\rangle_s
=
\mathcal A_t(\eta).
\end{align*}
This proves \eqref{eq:simple-bracket}.
\end{proof}

\begin{proposition}[BDG inequality for simple integrands]
\label{prop:simple-BDG}
Let $p\ge2$.  There exist constants $0<c_p\le C_p<\infty$, depending only
on $p$, such that for every
$\eta\in\mathbf M^{p,0}([0,T]\times\mathbb R)$,
\begin{align}
&c_p\underline{\sigma}^p
\hat{\mathbb E}
\left[
\left(
\int_0^T\int_{\mathbb R}
|\eta(s,x)|^2\,dxds
\right)^{p/2}
\right]
\le
\bar{\mathbb E}
\left[
\sup_{0\le t\le T}
|\mathcal{I}_t(\eta)|^p
\right]
\le
C_p\overline{\sigma}^p
\hat{\mathbb E}
\left[
\left(
\int_0^T\int_{\mathbb R}
|\eta(s,x)|^2\,dxds
\right)^{p/2}
\right].
\label{eq:simple-BDG}
\end{align}
\end{proposition}

\begin{proof}
Fix $P\in\mathcal P$.  By Lemma~\ref{le:simple-bracket} and the classical
BDG inequality,
\[
c_pE_P[\mathcal A_T(\eta)^{p/2}]
\le
E_P\left[
\sup_{0\le t\le T}|\mathcal{I}_t(\eta)|^p
\right]
\le
C_pE_P[\mathcal A_T(\eta)^{p/2}].
\]
The constants are independent of $P$. 
Observe that for every simple $\eta$,
\begin{equation}\label{eq:A-two-sided}
\underline{\sigma}^2
\int_0^T\int_{\mathbb R}|\eta(r,x)|^2\,dxdr
\le
\mathcal A_T(\eta)
\le
\overline{\sigma}^2
\int_0^T\int_{\mathbb R}|\eta(r,x)|^2\,dxdr.
\end{equation}
Indeed, after inserting $s$ and $t$ into the time partition of $\eta$,
this follows term by term from Proposition~\ref{prop:quad-var}. 
Taking the supremum over $P\in\mathcal P$ gives
\eqref{eq:simple-BDG}. 
\end{proof}

Notice that the maximal random variable is evaluated
with $\bar{\mathbb E}$; at this stage we do not assume that it already belongs
to $\mathbf L_G^1(\mathcal{F})$.

\begin{theorem}[BDG inequality]\label{th:BDG}
Let $p\ge2$, $T>0$, and
$\eta\in\mathbf H_G^p([0,T]\times\mathbb R)$. Then the stochastic
integral process
\[
\mathcal I_t(\eta)
=
\int_0^t\int_{\mathbb R}
\eta(s,x)\mathbf W(ds,dx),
\qquad 0\le t\le T,
\]
admits a quasi-surely continuous modification, still denoted by
$\mathcal I(\eta)$. For this modification, there exist constants
$0<c_p\le C_p<\infty$, depending only on $p$, such that
\begin{align}
&c_p\underline{\sigma}^p
\hat{\mathbb E}
\left[
\left(
\int_0^T\int_{\mathbb R}
|\eta(s,x)|^2\,dxds
\right)^{p/2}
\right]
\le
\bar{\mathbb E}
\left[
\sup_{0\le t\le T}
|\mathcal I_t(\eta)|^p
\right]
\le
C_p\overline{\sigma}^p
\hat{\mathbb E}
\left[
\left(
\int_0^T\int_{\mathbb R}
|\eta(s,x)|^2\,dxds
\right)^{p/2}
\right].
\label{eq:BDG}
\end{align}
\end{theorem}

\begin{proof}
Choose
$\{\eta_n\}_{n\ge1}
\subset\mathbf M^{p,0}([0,T]\times\mathbb R)$
such that
$\|\eta_n-\eta\|_{\mathbf H^p}\to 0 $ as $n\to\infty$, 
and write
$\mathcal I_t^n:=\mathcal I_t(\eta_n)$,
 $0\le t\le T.$
 By Proposition~\ref{prop:simple-BDG},
\begin{equation}\label{eq:BDG-Cauchy}
\bar{\mathbb E}
\left[
\sup_{0\le t\le T}
|\mathcal I_t^n-\mathcal I_t^m|^p
\right]
\le
C_p\overline{\sigma}^p
\|\eta_n-\eta_m\|_{\mathbf H^p}^p.
\end{equation}
Hence $\{\mathcal I^n\}_{n\ge1}$ is a Cauchy sequence in the maximal
$L^p$ norm induced by $\bar{\mathbb E}$.

We now construct a quasi-surely continuous limit. Choose a subsequence
$\{\mathcal I^{n_k}\}_{k\ge1}$ such that
\[
\bar{\mathbb E}
\left[
\sup_{0\le t\le T}
|\mathcal I_t^{n_{k+1}}-\mathcal I_t^{n_k}|^p
\right]
\le
2^{-2kp},
\qquad k\ge1.
\]
Set
\[
D_k
:=
\sup_{t\in\mathbb Q\cap[0,T]}
|\mathcal I_t^{n_{k+1}}
-\mathcal I_t^{n_k}|,\qquad k\ge1.
\]
Each $\mathcal I^{n_k}$ has quasi-surely continuous paths, since it is
the stochastic integral of a simple integrand. Removing the countable
union of the corresponding polar exceptional sets, we may assume that
all the processes $\mathcal I^{n_k}$ are continuous. On this common
full-capacity set,
\[
D_k
=
\sup_{0\le t\le T}
|\mathcal I_t^{n_{k+1}}
-\mathcal I_t^{n_k}|.
\]

By the Markov inequality for the capacity $c$,
\[
c(D_k>2^{-k})
\le
2^{kp}\bar{\mathbb E}[D_k^p]
\le
2^{-kp}.
\]
Since $p\ge2$,
\[
\sum_{k=1}^{\infty}
c(D_k>2^{-k})<\infty.
\]
Then by the Borel-Cantelli Lemma under the capacity (see Lemma 6.1.9 in Peng \cite{P2019}),
\[
c\left(
\limsup_{k\to\infty}
\{D_k>2^{-k}\}
\right)=0.
\]
which means that $\limsup_{k\to\infty}
\{D_k>2^{-k}\} $ is  polar. Consequently, outside a polar set,
$D_k\le2^{-k}$
for all sufficiently large $k$. Hence
$\{\mathcal I^{n_k}\}$ is uniformly Cauchy on $[0,T]$ outside a polar
set and therefore converges uniformly  to a continuous process,
which we denote by $\widetilde{\mathcal I}$.

We next show that the whole sequence converges to
$\widetilde{\mathcal I}$ in the maximal $L^p$ norm. For every fixed $n$, by Fatou's
lemma under each $P\in\mathcal P$,
\[
E_P\left[
\sup_{0\le t\le T}
|\mathcal I_t^n-\widetilde{\mathcal I}_t|^p
\right]
\le
\liminf_{k\to\infty}
E_P\left[
\sup_{0\le t\le T}
|\mathcal I_t^n-\mathcal I_t^{n_k}|^p
\right].
\]
Taking the supremum over $P\in\mathcal P$ yields
\begin{align*}
\bar{\mathbb E}\left[
\sup_{0\le t\le T}
|\mathcal I_t^n-\widetilde{\mathcal I}_t|^p
\right]
&\le
\liminf_{k\to\infty}
\bar{\mathbb E}\left[
\sup_{0\le t\le T}
|\mathcal I_t^n-\mathcal I_t^{n_k}|^p
\right]
\\
&\le
C_p\overline{\sigma}^p
\liminf_{k\to\infty}
\|\eta_n-\eta_{n_k}\|_{\mathbf H^p}^p
\\
&=
C_p\overline{\sigma}^p
\|\eta_n-\eta\|_{\mathbf H^p}^p.
\end{align*}
Hence
\begin{equation}\label{eq:maximal-convergence}
\lim_{n\to\infty}\bar{\mathbb E}\left[
\sup_{0\le t\le T}
|\mathcal I_t^n-\widetilde{\mathcal I}_t|^p
\right]
=0.
\end{equation}

It remains to identify $\widetilde{\mathcal I}$ with the stochastic
integral process already defined by completion in $M_G^2([0,T]\times\mathbb R)$. Since $p\ge2$ and 
$\mathbf H_G^p([0,T]\times\mathbb R)
\subset
\mathbf M_G^2([0,T]\times\mathbb R)$,
 we have $\eta_n\rightarrow\eta$
in $\mathbf M_G^2([0,T]\times\mathbb R)$
as $n\to\infty$.
By the continuity of the stochastic integral from
$\mathbf{M}_G^2([0,T]\times\mathbb R)$ into $\mathbf L_G^2(\mathcal{F})$, for every fixed $t\in[0,T]$,
\[
\mathcal I_t^n
\longrightarrow
\mathcal I_t(\eta)
=
\int_0^t\int_{\mathbb R}
\eta(s,x)\mathbf W(ds,dx)
\qquad\text{in }\mathbf L_G^2.
\]
In particular,
$\mathcal I_t^n\to\mathcal I_t(\eta)$ in capacity. On the other hand,
\eqref{eq:maximal-convergence} implies
\[
\mathcal I_t^n
\longrightarrow
\widetilde{\mathcal I}_t
\qquad\text{in capacity}
\]
for every fixed $t$. By uniqueness of limits in capacity,
\[
\widetilde{\mathcal I}_t
=
\mathcal I_t(\eta)
\qquad\text{q.s.}
\]
for every fixed $t$. Thus $\widetilde{\mathcal I}$ is a quasi-surely
continuous modification of the stochastic integral process. We henceforth
denote this modification again by $\mathcal I(\eta)$.

Finally, we pass the simple integrand BDG inequality to the limit.
Taking $p$th roots in \eqref{eq:simple-BDG} gives
\begin{equation}\label{ineq-1}
c_p^{1/p}\underline{\sigma}
\|\eta_n\|_{\mathbf H^p}
\le
\left\{
\bar{\mathbb E}
\left[
\sup_{0\le t\le T}
|\mathcal I_t^n|^p
\right]
\right\}^{1/p}
\le
C_p^{1/p}\overline{\sigma}
\|\eta_n\|_{\mathbf H^p}.
\end{equation}
As $n\to \infty$, since
\[
\bigl|
\|\eta_n\|_{\mathbf H^p}
-
\|\eta\|_{\mathbf H^p}
\bigr|
\le
\|\eta_n-\eta\|_{\mathbf H^p}
\longrightarrow0,
\]
and, by the triangle inequality for the robust maximal $L^p$ norm,
\[
\begin{aligned}
&\left|
\left\{
\bar{\mathbb E}
\left[
\sup_{0\le t\le T}
|\mathcal I_t^n|^p
\right]
\right\}^{1/p}
-
\left\{
\bar{\mathbb E}
\left[
\sup_{0\le t\le T}
|\mathcal I_t(\eta)|^p
\right]
\right\}^{1/p}
\right|
\le
\left\{
\bar{\mathbb E}
\left[
\sup_{0\le t\le T}
|\mathcal I_t^n-\mathcal I_t(\eta)|^p
\right]
\right\}^{1/p}
\longrightarrow0,
\end{aligned}
\]
we may pass to the limit in the preceding two-sided inequality  \eqref{ineq-1}. Raising
the resulting inequality to the $p$th power proves
\eqref{eq:BDG}.
\end{proof}

\begin{corollary}\label{cor:BDG-fixed}
Let $p\ge2$ and
$\eta\in\mathbf H_G^p([0,T]\times\mathbb R)$. Then
$\mathcal I_T(\eta)\in\mathbf L_G^p(\mathcal F_T)$ and
\begin{equation}\label{eq:BDG-fixed}
\hat{\mathbb E}
\left[
|\mathcal I_T(\eta)|^p
\right]
\le
C_p\overline{\sigma}^p
\hat{\mathbb E}
\left[
\left(
\int_0^T\int_{\mathbb R}
|\eta(s,x)|^2\,dxds
\right)^{p/2}
\right].
\end{equation}
\end{corollary}
\begin{proof}
Let $\eta_n\in\mathbf M^{p,0}([0,T]\times\mathbb R)$ be the simple
approximations used in the proof of Theorem~\ref{th:BDG}. For $m,n\ge1$,
the upper bound in Proposition~\ref{prop:simple-BDG} gives
\[
\begin{aligned}
\|\mathcal I_T(\eta_n)-\mathcal I_T(\eta_m)\|_{\mathbf L_G^p}^p
&=\hat{\mathbb E}\!
\left[|\mathcal I_T(\eta_n-\eta_m)|^p\right]\\
&\le
\bar{\mathbb E}\!
\left[\sup_{0\le t\le T}|\mathcal I_t(\eta_n-\eta_m)|^p\right]\\
&\le
C_p\overline\sigma^p
\|\eta_n-\eta_m\|_{\mathbf H^p}^p.
\end{aligned}
\]
Hence $\{\mathcal I_T(\eta_n)\}$ is Cauchy in the complete space
$\mathbf L_G^p(\mathcal F_T)$; let its limit be $Y$.
Letting $n\to\infty$,
since $p\ge2$ and $\mathbf H_G^p([0,T]\times\mathbb R)\subset\mathbf M_G^2([0,T]\times\mathbb R)$, we also have
$\eta_n\to\eta$ in $\mathbf M_G^2([0,T]\times\mathbb R)$ and then
\[
\mathcal I_T(\eta_n)\longrightarrow\mathcal I_T(\eta)
\qquad\text{in }\mathbf L_G^2(\mathcal F_T).
\]
On the other hand, convergence to $Y$ in $\mathbf L_G^p$ implies
convergence to $Y$ in $\mathbf L_G^2$. Therefore
$Y=\mathcal I_T(\eta)$ in $\mathbf L_G^2(\mathcal F_T)$, and hence quasi-surely. This
proves $\mathcal I_T(\eta)\in\mathbf L_G^p(\mathcal F_T)$.
Finally, letting $n\to\infty$ in the corresponding fixed-time estimate for
$\eta_n$ gives \eqref{eq:BDG-fixed}.
\end{proof}
\begin{remark}
The use of $\bar{\mathbb E}$ in the maximal terms of
\eqref{eq:simple-BDG} and \eqref{eq:BDG} is deliberate.
The representation
$\hat{\mathbb E}[X]
=
\sup_{P\in\mathcal P}E_P[X]
$
is initially available for $X\in\mathbf L_G^1(\mathcal{F})$, whereas, before the
BDG inequality is established, it is not yet known that
$\sup_{0\le t\le T}
|\mathcal I_t(\eta)|^p$
belongs to $\mathbf L_G^1(\mathcal{F})$. For a nonnegative Borel measurable random
variable $X$, however, the extended upper expectation
$\bar{\mathbb E}[X]
:=
\sup_{P\in\mathcal P}E_P[X]$
is well defined. Thus the use of $\bar{\mathbb E}$ avoids any circularity.
Whenever $X\in\mathbf L_G^1(\mathcal{F})$, one has
$\bar{\mathbb E}[X]=\hat{\mathbb E}[X]$.
\end{remark}

\section{Stochastic wave equations driven by space-time $G$-white noise}

We now apply the stochastic integration theory and the BDG inequality  to study one-dimensional stochastic wave equations under sublinear expectation. 
Throughout this section,
$L>0$ is fixed and the equations are considered on
$[0,\infty)\times(0,L)$ with homogeneous Dirichlet boundary conditions. Whenever
a finite time horizon is needed for an estimate or a functional space argument,
$T>0$ is arbitrary. 

Our treatment follows Walsh's random field approach for the classical stochastic wave
equation. 
In particular, we use the deterministic Green kernel estimates from
\cite[Secs.~3.4, 4.3.2 and Appendix~B.8]{DS}. For notational convenience,
$\mathbf S_G^p([0,T]\times[0,L])$ and $\mathbf H_G^p([0,T]\times[0,L])$ 
denote the corresponding spaces obtained by
restricting the spatial variable to $[0,L]$; stochastic integrands are extended
by zero outside $[0,L]$ whenever necessary.

\subsection{The Dirichlet wave kernel and convolution estimates}

Let $G_L$ denote the Green kernel of the one-dimensional wave operator $\partial_{tt}-\partial_{xx}$ on
$(0,L)$ with homogeneous Dirichlet boundary conditions. By Duffy \cite{Duffy} and Dalang and  Sanz-Sol\'e \cite{DS}, 
\begin{align}
G_L(t,x,y)&=\frac{1}{2}\sum_{m=-\infty}^{\infty}\left[\mathbf{1}_{|x-y-2mL|\leq t}-\mathbf{1}_{|x+y-2mL|\leq t}\right]\label{eq:GL-image}\\
&=\sum_{m=1}^\infty\frac{2}{\pi m}\sin\left(\frac{m\pi x}{L}\right)\sin\left(\frac{m\pi y}{L}\right)\sin\left(\frac{m\pi t}{L}\right), \quad t\geq0, x,y\in[0,L].\label{eq:GL-series}
\end{align} 
Then for any fixed $(t,x)\in(0,\infty)\times[0,L]$, it is easy to check that $G_L(t-\cdot,x,\cdot)\in L^2([0,t]\times[0,L])$. Fix an arbitrary $T>0$. By Lemma B.8.1 in Dalang and Sanz-Sol\'e \cite{DS}, for
$0\le s\le t\le T$ and $x,y\in[0,L]$,
\begin{align}
\int_0^t\int_0^L
|G_L(t-r,x,z)-G_L(t-r,y,z)|^2\,dzdr
&\le t|x-y|,
\label{eq:est-g-x}\\
\int_0^s\int_0^L
|G_L(t-r,x,z)-G_L(s-r,x,z)|^2\,dzdr
&\le 2s|t-s|,
\label{eq:est-g-t}\\
\int_s^t\int_0^L
|G_L(t-r,x,z)|^2\,dzdr
&\le L|t-s|.
\label{eq:est-g-t2}
\end{align}

 For $(t,x)\in[0,T]\times[0,L]$, define the deterministic space-time kernel
\begin{equation}
\label{eq:Ktx}
K_{t,x}(r,z)
:=
\mathbf 1_{[0,t)}(r)G_L(t-r;x,z),
\qquad (r,z)\in[0,T]\times[0,L].
\end{equation}
We have the following useful  estimates.
\begin{lemma}
\label{lem:wave-kernel}
For every $T>0$, there exists $C=C(T,L)>0$ such that, for all
$(t,x),(s,y)\in[0,T]\times[0,L]$,
\begin{align}
\sup_{r\ge0}\sup_{x\in[0,L]}
\int_0^L |G_L(r;x,z)|^2dz
&\le \frac{L}{3},
\label{eq:GL-uniform-L2}\\
\|K_{t,x}-K_{s,y}\|_{L^2([0,T]\times[0,L])}^2
&\le C\bigl(|t-s|+|x-y|\bigr),
\label{eq:K-L2-inc}\\
\|K_{t,x}-K_{s,y}\|_{L^1([0,T]\times[0,L])}
&\le C\bigl(|t-s|+|x-y|\bigr).
\label{eq:K-L1-inc}
\end{align}
\end{lemma}

\begin{proof}{}
The $L^2$ increment estimate \eqref{eq:K-L2-inc} follows directly from
\eqref{eq:est-g-x}, \eqref{eq:est-g-t} and \eqref{eq:est-g-t2}. The uniform bound \eqref{eq:GL-uniform-L2} follows from the sine series
representation \eqref{eq:GL-series} and Parseval's identity:
\[
\int_0^L |G_L(r;x,z)|^2dz
=
\frac{2L}{\pi^2}\sum_{m=1}^{\infty}\frac{1}{m^2}
\sin^2\frac{m\pi x}{L}\sin^2\frac{m\pi r}{L}\le
\frac{2L}{\pi^2}
\sum_{m=1}^{\infty}\frac1{m^2}
=
\frac{L}{3}.
\]
Finally, the image representation \eqref{eq:GL-image} shows that, away from a
Lebesgue-null set of characteristic boundaries, $K_{t,x}$ takes values in
$\{0,-\tfrac12,\tfrac12\}$. Hence
$|K_{t,x}-K_{s,y}|\le2|K_{t,x}-K_{s,y}|^2$ a.e., and
\eqref{eq:K-L1-inc} follows from \eqref{eq:K-L2-inc}.
\end{proof}

\begin{lemma}
\label{lem:weighted-wave}
Let $p\ge2$, $T>0$, and
$\eta\in\mathbf S_G^p([0,T]\times[0,L])$.
\begin{enumerate}
\item[(i)] If $h\in L^2([0,T]\times[0,L])$ is deterministic, then
$h\eta\in\mathbf H_G^p([0,T]\times[0,L])$ and
\begin{equation}
\label{eq:h-eta-Hp}
\hat{\mathbb E}
\left[
\left(
\int_0^T\int_0^L
h(r,z)^2|\eta(r,z)|^2\,dzdr
\right)^{p/2}
\right]
\le
\|h\|_{L^2}^p
\|\eta\|_{\mathbf S^p}^p.
\end{equation}
Consequently,
\begin{equation}
\label{eq:weighted-stochastic-wave}
\left\|
\int_0^T\int_0^L
h(r,z)\eta(r,z)\mathbf W(dr,dz)
\right\|_{\mathbf L_G^p}
\le
C_p^{1/p}\overline\sigma
\|h\|_{L^2}\|\eta\|_{\mathbf S^p}.
\end{equation}

\item[(ii)] If $h\in L^1([0,T]\times[0,L])$ is deterministic, then the Bochner
integral of $h\eta$ is well defined in
$\mathbf L_G^p(\mathcal F_T)$ and
\begin{equation}
\label{eq:weighted-deterministic-wave}
\left\|
\int_0^T\int_0^L
h(r,z)\eta(r,z)\,dzdr
\right\|_{\mathbf L_G^p}
\le
\|h\|_{L^1}\|\eta\|_{\mathbf S^p}.
\end{equation}
\end{enumerate}
\end{lemma}

\begin{proof}
(i) First, for $p\geq2$,
suppose  that $\eta
\in\mathbf M^{p,0}([0,T]\times [0,L])$ is a simple  random field and that $h$ is
a deterministic simple function. By taking a common refinement of their
time-space partitions, $h\eta$ belongs to $\mathbf M^{p,0}([0,T]\times[0,L])$.

By H\"older's inequality,
\begin{align*}
\hat{\mathbb E}\left[
\left(
\int_0^T\int_0^L
h(r,z)^2|\eta(r,z)|^2\,dzdr
\right)^{p/2}
\right]
&\le \hat{\mathbb{E}}\left[\left(
\int_0^T\int_0^L h(r,z)^2\,dzdr
\right)^{p/2-1}
\int_0^T\int_0^L
h(r,z)^2|\eta(r,z)|^p\,dzdr\right]\\
&\le \left(
\int_0^T\int_0^L h(r,z)^2\,dzdr
\right)^{p/2-1}
\int_0^T\int_0^L
h(r,z)^2\hat{\mathbb{E}}\left[|\eta(r,z)|^p\right]\,dzdr\\
&\le
\|h\|_{L^2}^{p-2}
\|h\|_{L^2}^{2}
\|\eta\|_{\mathbf S^p}^p
=
\|h\|_{L^2}^{p}
\|\eta\|_{\mathbf S^p}^p.
\end{align*}
Equivalently,
\begin{equation}
\label{eq:simple-H-norm}
\|h\eta\|_{\mathbf H_G^p}
\le
\|h\|_{L^2}\|\eta\|_{\mathbf S^p}.
\end{equation}

Next, for $h\in L^2([0,T]\times[0,L])$ and  $\eta\in\mathbf S_G^p([0,T]\times[0,L])$, there exist deterministic simple functions $\{h_n\}_{n=1}^\infty$ and $\{\eta_n\}_{n=1}^\infty\subset\mathbf M^{p,0}([0,T]\times[0,L])$ such that $\|h_n-h\|_{L^2}\rightarrow0$ and $\|\eta_n-\eta\|_{\mathbf S^p}\rightarrow0 $ as $n\to\infty$.
After taking common refinements of the underlying time-space partitions,
$h_n\eta_n$ belongs to 
$\mathbf{M}^{p,0}([0,T]\times[0,L])$.

For $m,n\ge1$, by \eqref{eq:simple-H-norm},
\begin{align*}
\|h_n\eta_n-h_m\eta_m\|_{\mathbf H_G^p}
&\le
\|(h_n-h_m)\eta_n\|_{\mathbf H_G^p}
+
\|h_m(\eta_n-\eta_m)\|_{\mathbf H_G^p}
\\
&\le
\|h_n-h_m\|_{L^2}\|\eta_n\|_{\mathbf S^p}
+
\|h_m\|_{L^2}
\|\eta_n-\eta_m\|_{\mathbf S^p}.
\end{align*}
Since $\{\eta_n\}$ is bounded in $\mathbf S_G^p$ and $\{h_n\}$ is
bounded in $L^2$, the right-hand side tends to zero as $m,n\to\infty$.
Thus $\{h_n\eta_n\}$ is a Cauchy sequence in $\mathbf H_G^p([0,T]\times[0,L])$.

To identify the limit with the pointwise product $h\eta$, fix
$P\in\mathcal P$. By the classical weighted $L^p(P)$ estimate,
\begin{align*}
& E_P\left[
\left(
\int_0^T\int_0^L
|h_n(r,z)\eta_n(r,z)-h(r,z)\eta(r,z)|^2\,dzdr
\right)^{p/2}
\right]^{1/p}
\\
&\qquad\le
\|h_n-h\|_{L^2}\|\eta_n\|_{\mathbf S^p}
+
\|h\|_{L^2}\|\eta_n-\eta\|_{\mathbf S^p}.
\end{align*}
Taking the supremum over $P\in\mathcal P$ shows that the left-hand side
converges to zero in the robust $\mathbf H^p$ norm. Hence the completion
limit of $h_n\eta_n$ is represented by the pointwise product $h\eta$, and
therefore
$h\eta\in\mathbf H_G^p([0,T]\times[0,L])$.

By continuity of the norm and
\eqref{eq:simple-H-norm},
\begin{align*}
\|h\eta\|_{\mathbf H^p}
&=
\lim_{n\to\infty}
\|h_n\eta_n\|_{\mathbf H^p}
\le
\lim_{n\to\infty}
\|h_n\|_{L^2}\|\eta_n\|_{\mathbf S^p}=
\|h\|_{L^2}\|\eta\|_{\mathbf S^p}.
\end{align*}
Raising both sides to the power $p$ gives
\eqref{eq:h-eta-Hp}.

Since $h\eta\in\mathbf H_G^p([0,T]\times[0,L])$, the stochastic integral in the left-hand side of \eqref{eq:weighted-stochastic-wave} is well defined.
By the  BDG inequality,
\begin{align*}
&
\hat{\mathbb E}
\left[
\left|
\int_0^T\int_0^L
h(r,z)\eta(r,z)\mathbf W(dr,dz)
\right|^p
\right]
\\
&\qquad\le
C_p\overline\sigma^p
\hat{\mathbb E}
\left[
\left(
\int_0^T\int_0^L
h(r,z)^2|\eta(r,z)|^2\,dzdr
\right)^{p/2}
\right]
\\
&\qquad\le
C_p\overline\sigma^p
\|h\|_{L^2}^p
\|\eta\|_{\mathbf S^p}^p.
\end{align*}
Taking the $p$th root proves
\eqref{eq:weighted-stochastic-wave}.

(ii) Let $h\in L^1([0,T]\times[0,L])$ and
$\eta\in\mathbf S_G^p([0,T]\times[0,L])$.
Choose deterministic simple functions $h_n$ and  simple random
fields $\eta_n$ such that
$\|h_n-h\|_{L^1}\to0,
\|\eta_n-\eta\|_{\mathbf S^p}\to0$ as $n\to \infty$.
After taking common refinements of the underlying partitions,
$h_n\eta_n\in\mathbf M^{1,0}([0,T]\times[0,L])$.

For $m,n\ge1$,
\begin{align*}
\|h_n\eta_n-h_m\eta_m\|_{\mathbf M^1}
&\le
\|h_n-h_m\|_{L^1}\|\eta_n\|_{\mathbf S^p}
+
\|h_m\|_{L^1}\|\eta_n-\eta_m\|_{\mathbf S^p}.
\end{align*}
Hence $\{h_n\eta_n\}$ is a Cauchy sequence in $\mathbf M_G^1([0,T]\times[0,L])$, and therefore
converges to some $\xi\in\mathbf M_G^1([0,T]\times[0,L])$.
The preceding estimates also show that this limit coincides with the
pointwise product $h\eta$. Thus
$h\eta\in\mathbf M_G^1([0,T]\times[0,L])$, and the deterministic
integral introduced in Section~2 is well defined in
$\mathbf L_G^1(\mathcal F_T)$.

Set
\[
Y_n
:=
\int_0^T\int_0^L
h_n(r,z)\eta_n(r,z)\,dzdr,\quad n\ge1,
\]
and 
\[
\|Y_n\|_{\mathbf L_G^p}
\le
\|h_n\|_{L^1}\|\eta_n\|_{\mathbf S^p}.
\]
Similarly,
\begin{align*}
\|Y_n-Y_m\|_{\mathbf L_G^p}
&\le
\|h_n-h_m\|_{L^1}\|\eta_n\|_{\mathbf S^p}
+
\|h_m\|_{L^1}
\|\eta_n-\eta_m\|_{\mathbf S^p}.
\end{align*}
Thus $\{Y_n\}$ is Cauchy in
$\mathbf L_G^p(\mathcal F_T)$, and hence converges to some
$Y\in\mathbf L_G^p(\mathcal F_T)$.

Since convergence in $\mathbf L_G^p$ implies convergence in
$\mathbf L_G^1$, while $h_n\eta_n\to h\eta$ in $\mathbf M_G^1$,
the continuity of the deterministic integral defined in Section~2
shows that
\[
Y=
\int_0^T\int_0^L h(r,z)\eta(r,z)\,dzdr.
\]
Consequently,
\[
\left\|
\int_0^T\int_0^L
h(r,z)\eta(r,z)\,dzdr
\right\|_{\mathbf L_G^p}
\le
\|h\|_{L^1}\|\eta\|_{\mathbf S^p},
\]
which proves \eqref{eq:weighted-deterministic-wave}.
\end{proof}

For $\eta\in\mathbf S_G^p([0,T]\times[0,L])$, $(t,x)\in[0,T]\times[0,L]$, define
\begin{align}
(\mathcal S\eta)(t,x)
&:=
\int_0^t\int_0^L
G_L(t-s;x,y)\eta(s,y)\mathbf W(ds,dy),
\label{eq:stochastic-convolution}\\
(\mathcal D\eta)(t,x)
&:=
\int_0^t\int_0^L
G_L(t-s;x,y)\eta(s,y)dyds.
\label{eq:deterministic-convolution}
\end{align}

\begin{proposition}
\label{prop:wave-convolution}
Let $p\ge2$. Then $\mathcal S$ and $\mathcal D$ map
$\mathbf S_G^p([0,T]\times[0,L])$ into itself. Moreover, for some
$C=C(p,T,L,\overline\sigma)>0$ and all
$(t,x),(s,y)\in[0,T]\times[0,L]$,
\begin{align}
\|(\mathcal S\eta)(t,x)-(\mathcal S\eta)(s,y)\|_{\mathbf L_G^p}
&\le
C\|\eta\|_{\mathbf S^p}
\left(|t-s|^{1/2}+|x-y|^{1/2}\right),
\label{eq:S-joint-inc}\\
\|(\mathcal D\eta)(t,x)-(\mathcal D\eta)(s,y)\|_{\mathbf L_G^p}
&\le
C\|\eta\|_{\mathbf S^p}
\left(|t-s|+|x-y|\right).
\label{eq:D-joint-inc}
\end{align}
\end{proposition}

\begin{proof}
For $(t,x)\in[0,T]\times[0,L]$, recall that
\[
K_{t,x}(r,z)
=
\mathbf 1_{[0,t)}(r)G_L(t-r;x,z),
\qquad
(r,z)\in[0,T]\times[0,L].
\]

We first verify that the two convolutions are well defined.
By \eqref{eq:GL-uniform-L2},
\[
\|K_{t,x}\|_{L^2([0,T]\times[0,L])}^2
=
\int_0^t\int_0^L
|G_L(t-r;x,z)|^2\,dzdr
\le
\frac{LT}{3}.
\]
Hence $K_{t,x}\in L^2([0,T]\times[0,L])$. By
Lemma~\ref{lem:weighted-wave}(i),
$K_{t,x}\eta\in\mathbf H_G^p([0,T]\times[0,L])$.
Since $K_{t,x}$ is supported on $[0,t)\times[0,L]$, the stochastic
integral may equivalently be constructed on the time horizon $[0,t]$;
therefore
$(\mathcal S\eta)(t,x)\in\mathbf L_G^p(\mathcal F_t)$.


Similarly, since the  space-time cylinder has finite measure,
\[
\|K_{t,x}\|_{L^1}
\le
(TL)^{1/2}\|K_{t,x}\|_{L^2}<\infty.
\]
As $K_{t,x}$ is supported on $[0,t)\times[0,L]$, the same
construction in Lemma~\ref{lem:weighted-wave}(ii) may be carried out on
the time horizon $[0,t]$. Consequently,
$(\mathcal D\eta)(t,x)\in\mathbf L_G^p(\mathcal F_t)$.

We next prove the increment estimates. By the linearity of the stochastic
integral,
\begin{align*}
(\mathcal S\eta)(t,x)-(\mathcal S\eta)(s,y)
&=
\int_0^T\int_0^L
\bigl(K_{t,x}(r,z)-K_{s,y}(r,z)\bigr)
\eta(r,z)\mathbf W(dr,dz).
\end{align*}
Applying \eqref{eq:weighted-stochastic-wave} with
$h=K_{t,x}-K_{s,y}$ yields
\begin{align*}
\|(\mathcal S\eta)(t,x)-(\mathcal S\eta)(s,y)\|_{\mathbf L_G^p}
&\le
C_p^{1/p}\overline\sigma
\|K_{t,x}-K_{s,y}\|_{L^2}
\|\eta\|_{\mathbf S^p}.
\end{align*}
By \eqref{eq:K-L2-inc},
\[
\|K_{t,x}-K_{s,y}\|_{L^2}
\le
C
\bigl(|t-s|+|x-y|\bigr)^{1/2}
\le
C\left(
|t-s|^{1/2}+|x-y|^{1/2}
\right).
\]
Consequently,
\[
\|(\mathcal S\eta)(t,x)-(\mathcal S\eta)(s,y)\|_{\mathbf L_G^p}
\le
C\|\eta\|_{\mathbf S^p}
\left(
|t-s|^{1/2}+|x-y|^{1/2}
\right),
\]
which proves \eqref{eq:S-joint-inc}.

Similarly, by the linearity of the Bochner integral,
\[
(\mathcal D\eta)(t,x)-(\mathcal D\eta)(s,y)
=
\int_0^T\int_0^L
(K_{t,x}-K_{s,y})(r,z)\eta(r,z)\,dzdr.
\]
Hence, by \eqref{eq:weighted-deterministic-wave} and \eqref{eq:K-L1-inc},
\begin{align*}
\|(\mathcal D\eta)(t,x)-(\mathcal D\eta)(s,y)\|_{\mathbf L_G^p}
&\le
\|K_{t,x}-K_{s,y}\|_{L^1}
\|\eta\|_{\mathbf S^p}
\\
&\le
C\|\eta\|_{\mathbf S^p}
\bigl(|t-s|+|x-y|\bigr),
\end{align*}
which proves \eqref{eq:D-joint-inc}.

It remains to prove that the two random fields $\mathcal S\eta$ and $\mathcal D\eta$ belong to
$\mathbf S_G^p([0,T]\times[0,L])$.
For $n\ge1$, let
$t_i^n=\frac{iT}{n}$,
 and
$x_j^n=\frac{jL}{n}$,
$0\le i,j\le n$,
and let $I_i^n=[t_i^n,t_{i+1}^n)$ be the time partition intervals,
with the value at $t=T$ understood according to the $\mathbf S^p$ terminal-value
convention introduced in Section~2. Let $J_j^n$ be the corresponding
spatial partition intervals, with the last spatial interval closed at
$x=L$. Define
\[
S_n(t,x)
=
\sum_{i=0}^{n-1}\sum_{j=0}^{n-1}
(\mathcal S\eta)(t_i^n,x_j^n)
\mathbf1_{I_i^n}(t)\mathbf1_{J_j^n}(x).
\]
Since
$
(\mathcal S\eta)(t_i^n,x_j^n)
\in\mathbf L_G^p(\mathcal F_{t_i^n}),
$
$\{S_n\}_{n=1}^\infty\subset \mathbf{M}^{p,0}([0,T]\times[0,L])$ are simple random fields.

If $(t,x)\in I_i^n\times J_j^n$, then
\[
|t-t_i^n|\le\frac{T}{n},
\qquad
|x-x_j^n|\le\frac{L}{n}.
\]
Therefore, by \eqref{eq:S-joint-inc},
\begin{align*}
\|(\mathcal S\eta)(t,x)-S_n(t,x)\|_{\mathbf L_G^p}
&\le
C\|\eta\|_{\mathbf S^p}
\left[
\left(\frac{T}{n}\right)^{1/2}
+
\left(\frac{L}{n}\right)^{1/2}
\right].
\end{align*}
Taking the supremum over $(t,x)\in[0,T]\times[0,L]$ gives
\[
\sup_{(t,x)\in[0,T]\times[0,L]}
\|(\mathcal S\eta)(t,x)-S_n(t,x)\|_{\mathbf L_G^p}
\le
C\|\eta\|_{\mathbf S^p}
(T^{1/2}+L^{1/2})n^{-1/2}
\longrightarrow0, \quad\text{as } n\to\infty.
\]
Since $S_n\in\mathbf M^{p,0}([0,T]\times[0,L])$, it follows from the
definition of $\mathbf S_G^p$ as the completion of the simple fields
under the supremum norm that
$\mathcal S\eta\in\mathbf S_G^p([0,T]\times[0,L])$.

Likewise, define
\[
D_n(t,x)
=
\sum_{i=0}^{n-1}\sum_{j=0}^{n-1}
(\mathcal D\eta)(t_i^n,x_j^n)
\mathbf1_{I_i^n}(t)\mathbf1_{J_j^n}(x).
\]
Since
$(\mathcal D\eta)(t_i^n,x_j^n)
\in\mathbf L_G^p(\mathcal F_{t_i^n})$,
$D_n$ is  simple random field belonging to $\mathbf{M}^{p,0}([0,T]\times[0,L])$. By \eqref{eq:D-joint-inc},
\[
\sup_{(t,x)\in[0,T]\times[0,L]}
\|(\mathcal D\eta)(t,x)-D_n(t,x)\|_{\mathbf L_G^p}
\le
C\|\eta\|_{\mathbf S^p}
(T+L)n^{-1}
\longrightarrow0 \quad\text{as } n\to\infty.
\]
Therefore
$
\mathcal D\eta
\in
\mathbf S_G^p([0,T]\times[0,L]).
$
The proof is complete.
\end{proof}

\subsection{Semilinear $G$-stochastic wave equations}

Consider a semilinear stochastic wave equation driven by the multiplicative space-time $G$-white noise on a bounded interval:
\begin{equation}\label{eq:nswe}
\begin{cases}
  \frac{\partial^2}{\partial t^2}u(t,x)-\frac{\partial^2}{\partial x^2}u(t,x)=b(u(t,x))+a(u(t,x))\dot{\mathbf{W}}(t,x),\quad t>0,0<x<L,\\
  u(0,x)=u_0(x),\quad  0\le x\le L,\\
  \frac{\partial}{\partial t}u(0,x)=v_0(x),\quad  0<x<L,\\
  u(t,0)=u(t,L)=0,\quad t\ge 0,
\end{cases}
\end{equation}
where $\dot{\mathbf{W}}(t,x)$
is the generalized derivative of the space-time $G$-white noise $\mathbf{W}(t,x)$ introduced in \cite{JP25}. 

Assume that
\begin{equation}
\label{eq:wave-initial-assumption}
u_0\in C([0,L]),\qquad u_0(0)=u_0(L)=0,
\qquad v_0\in L^1(0,L),
\end{equation}
and suppose that $a,b:\mathbb R\to\mathbb R$ are globally Lipschitz functions, that is, there exists a constant $C$ such that
\begin{equation}\label{eq:Lipschitz-ab}
|a(x)-a(y)|+|b(x)-b(y)|\leq C|x-y|,\ \forall x,y\in\mathbb{R}.
\end{equation} 
If $v_0$ is only in
$L^1(0,L)$, the initial velocity condition is understood in the usual weak
(distributional) sense.

Let $u_0^{\mathrm{op}}$ and $v_0^{\mathrm{op}}$ denote the odd
$2L$-periodic extensions of $u_0$ and $v_0$ to $\mathbb R$, respectively.
Define
\begin{equation}
\label{eq:I0}
I_0(t,x)
:=
\frac12\bigl(
u_0^{\mathrm{op}}(x+t)+u_0^{\mathrm{op}}(x-t)
\bigr)
+
\frac12\int_{x-t}^{x+t}v_0^{\mathrm{op}}(r)\,dr,
\qquad t\ge0,\ x\in[0,L].
\end{equation}
Then $I_0$ is bounded and continuous on every compact cylinder
$[0,T]\times[0,L]$, satisfies
\[
I_0(t,0)=I_0(t,L)=0,
\qquad t\ge0,
\]
and is the distributional solution of the homogeneous Dirichlet wave
equation with initial conditions $u_0$ and  $v_0$ (see \cite[Sec.~3.4.1]{DS}). 

By Duhamel's principle, the mild solution of $G$-stochastic wave equation \eqref{eq:nswe} is defined as follows.

\begin{definition}\label{def:nswe}
A spatio-temporal random field $u=\{u(t, x) : (t, x)\in [0,\infty) \times [0,L]\}$ is
said to be a mild solution of the semilinear $G$-stochastic wave equation \eqref{eq:nswe} if it satisfies the
following conditions:
\begin{itemize}
  \item[(i)] For every $T>0$, $(u(t,x))_{(t, x)\in [0,T] \times [0,L]}\in \mathbf{S}_G^2
([0, T] \times [0,L])$;
  \item[(ii)] For every $t\geq0$ and $x\in[0,L]$,
  \begin{align}\label{eq:nswe-df}
  u(t,x)=&I_0(t,x)
  +\int_0^t\int_0^LG_L(t-s,x,y)b(u(s,y))dyds+\int_0^t\int_0^LG_L(t-s,x,y)a(u(s,y)){\mathbf{W}}(ds,dy).
  \end{align}
  Equivalently,
\[
u=I_0+\mathcal D(b\circ u)+\mathcal S(a\circ u)
\]
on every finite time horizon.
\end{itemize}
\end{definition}

\begin{theorem}
\label{th:main}
Assume \eqref{eq:wave-initial-assumption} and \eqref{eq:Lipschitz-ab}. Then the semilinear $G$-stochastic wave equation
\eqref{eq:nswe} admits a unique global mild solution on
$[0,\infty)\times[0,L]$. Moreover, for every $T>0$ and $p\ge2$,
\begin{equation}
\label{eq:solution-Sp}
u|_{[0,T]\times[0,L]}
\in\mathbf S_G^p([0,T]\times[0,L]).
\end{equation}
\end{theorem}

\begin{proof}
Fix $T>0$ and $p\ge2$. Since $I_0$ is deterministic, bounded, and continuous
on $[0,T]\times[0,L]$, it 
belongs to $\mathbf S_G^p([0,T]\times[0,L])$. Define the Picard sequence on
$[0,T]\times[0,L]$ by
\[
u^{(0)}:=I_0,
\qquad
u^{(n+1)}
:=I_0+\mathcal D(b\circ u^{(n)})+\mathcal S(a\circ u^{(n)}),
\qquad n\ge0.
\]
Global Lipschitz continuity implies that $a\circ v,b\circ v\in\mathbf S_G^p([0,T]\times[0,L])$
whenever $v\in\mathbf S_G^p([0,T]\times[0,L])$. Hence Proposition~\ref{prop:wave-convolution}
shows inductively that the Picard iteration sequence is well defined in
$\mathbf S_G^p([0,T]\times[0,L])$.

For $n\ge1$, set
\[
D_n(t):=
\sup_{0\le x\le L}
\hat{\mathbb E}
\left[|u^{(n)}(t,x)-u^{(n-1)}(t,x)|^p\right].
\]
Using the Lipschitz property of $a,b$, H\"older's inequality, the
fixed-time BDG estimate \eqref{eq:BDG-fixed}, and
\eqref{eq:GL-uniform-L2}, we obtain, for $0\le t\le T$,
\[
D_{n+1}(t)
\le C\int_0^tD_n(s)\,ds,
\]
where $C=C(p,T,L,a,b,\overline\sigma)$ is independent of $n$. Moreover,
$D_1$ is bounded on $[0,T]$, because $I_0$ is bounded and $a,b$ have
linear growth. Let
$M:=\sup_{0\le t\le T}D_1(t)<\infty$. Iterating the preceding inequality
gives
\[
D_n(t)
\le
M\frac{(Ct)^{n-1}}{(n-1)!},
\qquad n\ge1.
\]
Consequently,
\[
\sum_{n=1}^\infty
\|u^{(n)}-u^{(n-1)}\|_{\mathbf S^p}
\le
M^{1/p}
\sum_{n=1}^\infty
\frac{(CT)^{(n-1)/p}}{[(n-1)!]^{1/p}}
<\infty.
\]
Thus $\{u^{(n)}\}$ is Cauchy in the complete space
$\mathbf S_G^p([0,T]\times[0,L])$. Let its limit be $u^{[p]}$.
The Lipschitz continuity of $a,b$ and Proposition~\ref{prop:wave-convolution}
allow us to pass to the limit in the Picard equation, so $u^{[p]}$
satisfies \eqref{eq:nswe-df}.

Taking $p=2$ gives a mild solution $u$. If $u$ and $v$ are two mild
solutions in $\mathbf S_G^2([0,T]\times[0,L])$, the same estimates give
\[
D(t):=
\sup_{0\le x\le L}\hat{\mathbb E}[|u(t,x)-v(t,x)|^2]
\le C\int_0^tD(s)\,ds.
\]
Gronwall's lemma yields $D\equiv0$, proving uniqueness. For $p>2$,
convergence in $\mathbf S_G^p$ implies convergence in $\mathbf S_G^2$;
therefore uniqueness identifies $u^{[p]}$ with the same solution $u$.
This proves \eqref{eq:solution-Sp}.

Finally, $T>0$ was arbitrary. If $0<T_1<T_2$, uniqueness implies that the
solution constructed on $[0,T_2]\times[0,L]$, restricted to
$[0,T_1]\times[0,L]$, coincides with the solution constructed on the
smaller cylinder. These consistent finite-horizon solutions define a
unique global mild solution on $[0,\infty)\times[0,L]$.
\end{proof}

Furthermore, we show that the mild solution \eqref{eq:nswe-df}  is also a weak solution to \eqref{eq:nswe}.

For $T>0$, set
\begin{equation}
\label{eq:test-class}
\mathcal T_T
:=
\left\{
\varphi\in C^\infty([0,T]\times[0,L]):
\begin{array}{l}
\varphi(T,\cdot)=\partial_t\varphi(T,\cdot)=0,
\varphi(\cdot,0)=\varphi(\cdot,L)=0
\end{array}
\right\}.
\end{equation}

\begin{definition}
\label{def:weak-wave}
A spatio-temporal random field
$u=\{u(t,x):(t,x)\in[0,\infty)\times[0,L]\}$ is called a weak
solution of the semilinear $G$-stochastic wave equation \eqref{eq:nswe}  if it satisfies the
following conditions:
\begin{itemize}
  \item[(i)] For every $T>0$,
$u|_{[0,T]\times[0,L]}\in\mathbf S_G^2([0,T]\times[0,L])$;
\item[(ii)] For each fixed
$t\ge0$,
$u(t,0)=u(t,L)=0$;
\item[(iii)] For every $T>0$ and every
$\varphi\in\mathcal T_T$,
\begin{align}
&\int_0^T\int_0^L
u(t,x)(\partial_{tt}\varphi-\partial_{xx}\varphi)(t,x)dxdt
+
\int_0^L u_0(x)\partial_t\varphi(0,x)dx
-
\int_0^L v_0(x)\varphi(0,x)dx
\notag\\
&\qquad=
\int_0^T\int_0^L b(u(t,x))\varphi(t,x)dxdt
+
\int_0^T\int_0^L a(u(t,x))\varphi(t,x)\mathbf W(dt,dx).
\label{eq:semilinear-weak}
\end{align}
\end{itemize}
\end{definition}

\begin{theorem}
\label{thm:semilinear-mild-weak}
Under the assumptions of Theorem~\ref{th:main}, the unique global mild
solution of \eqref{eq:nswe} is also a weak solution.
\end{theorem}
\begin{proof}
Fix $T>0$ and $\varphi\in\mathcal T_T$, and write
$q(t,x):=
(\partial_{tt}\varphi-\partial_{xx}\varphi)(t,x)$.
Since $u\in\mathbf S_G^2([0,T]\times[0,L])$ and $a,b$ are globally
Lipschitz, the definition of
$\mathbf S_G^2([0,T]\times[0,L])$ by completion of  simple random fields implies  that
$a\circ u,\ b\circ u
\in\mathbf S_G^2([0,T]\times[0,L]).$
Moreover, $\varphi$ and $q$ are bounded on the compact cylinder.
Hence all the deterministic and stochastic integrals  below are well defined in $\mathbf L_G^2(\mathcal F_T)$ by Lemma~\ref{lem:weighted-wave} and
Proposition~\ref{prop:wave-convolution}.

By the mild formulation,
\begin{equation}
\label{eq:mild-for-weak-proof}
u(t,x)
=
I_0(t,x)
+
(\mathcal D(b\circ u))(t,x)
+
(\mathcal S(a\circ u))(t,x).
\end{equation}
We treat the three terms on the right-hand side separately.

Since $I_0$ is the distributional solution of the homogeneous Dirichlet
wave equation with initial condition $(u_0,v_0)$,
\begin{equation}
\label{eq:I0-weak}
\int_0^T\int_0^L
I_0(t,x)q(t,x)\,dxdt
+
\int_0^L u_0(x)\partial_t\varphi(0,x)\,dx
-
\int_0^L v_0(x)\varphi(0,x)\,dx
=0.
\end{equation}

For the drift term, since $b\circ u\in\mathbf S_G^2([0,T]\times[0,L])$, the boundedness of
$q$ and the wave-kernel estimates imply that
\[
\mathbf1_{\{s<t\}}
G_L(t-s;x,y)b(u(s,y))q(t,x)
\]
is Bochner integrable as an
$\mathbf L_G^2(\mathcal F_T)$-valued function on
$[0,T]^2\times[0,L]^2$.
Hence, the Fubini theorem for Bochner integrals yields
\begin{align*}
&\int_0^T\int_0^L
(\mathcal D(b\circ u))(t,x)q(t,x)\,dxdt
\\
&=
\int_0^T\int_0^L
\int_0^t\int_0^L
G_L(t-s;x,y)b(u(s,y))q(t,x)
\,dyds\,dxdt
\\
&=
\int_0^T\int_0^L
b(u(s,y))
\left[
\int_s^T\int_0^L
G_L(t-s;x,y)q(t,x)\,dxdt
\right]dyds.
\end{align*}
Since
\begin{equation}\label{eq:Green-adjoint}
\int_s^T\int_0^L
G_L(t-s;x,y)q(t,x)\,dxdt
=
\varphi(s,y),
\end{equation}
we have
\begin{equation}
\label{eq:drift-fubini}
\int_0^T\int_0^L
(\mathcal D(b\circ u))(t,x)q(t,x)\,dxdt
=
\int_0^T\int_0^L
b(u(s,y))\varphi(s,y)\,dyds.
\end{equation}

For the stochastic convolution, consider the four-variable integrand
\[
\mathbf1_{\{s<t\}}
G_L(t-s;x,y)a(u(s,y))q(t,x).
\]
As $a\circ u\in\mathbf S_G^2([0,T]\times[0,L])$ and
$\mathbf1_{\{s<t\}}G_L(t-s;x,y)q(t,x)
\in
L^2([0,T]^2\times[0,L]^2)$,
approximating $a\circ u$ by  simple random fields and the deterministic
kernel by rectangular simple functions shows that this integrand belongs
to the four-variable completion required by the stochastic Fubini theorem
of \cite[Theorem~4.1]{JP25}. 
Then 
the stochastic Fubini theorem
 gives
\begin{align*}
&\int_0^T\int_0^L
(\mathcal S(a\circ u))(t,x)q(t,x)\,dxdt
\\
&\quad=
\int_0^T\int_0^L
a(u(s,y))
\left[
\int_s^T\int_0^L
G_L(t-s;x,y)q(t,x)\,dxdt
\right]\mathbf W(ds,dy).
\end{align*}
Applying \eqref{eq:Green-adjoint}, we obtain
\begin{equation}
\label{eq:stochastic-fubini-wave}
\int_0^T\int_0^L
(\mathcal S(a\circ u))(t,x)q(t,x)\,dxdt
=
\int_0^T\int_0^L
a(u(s,y))\varphi(s,y)\mathbf W(ds,dy).
\end{equation}

Multiplying \eqref{eq:mild-for-weak-proof} by $q(t,x)$, integrating over
$[0,T]\times[0,L]$, and combining
\eqref{eq:I0-weak}, \eqref{eq:drift-fubini}, and
\eqref{eq:stochastic-fubini-wave}, we obtain
\begin{align*}
&\int_0^T\int_0^L
u(t,x)
(\partial_{tt}\varphi-\partial_{xx}\varphi)(t,x)\,dxdt
+
\int_0^L u_0(x)\partial_t\varphi(0,x)\,dx
-
\int_0^L v_0(x)\varphi(0,x)\,dx
\\
&\qquad=
\int_0^T\int_0^L
b(u(t,x))\varphi(t,x)\,dxdt
+
\int_0^T\int_0^L
a(u(t,x))\varphi(t,x)\mathbf W(dt,dx),
\end{align*}
in $\mathbf L_G^2(\mathcal F_T)$.

Finally, by construction,
\[
I_0(t,0)=I_0(t,L)=0,
\]
and the Dirichlet Green kernel satisfies
\[
G_L(r;0,y)=G_L(r;L,y)=0,
\qquad r>0.
\]
Hence both convolution terms in \eqref{eq:mild-for-weak-proof} vanish at
$x=0$ and $x=L$. Consequently, for each fixed $t\ge0$,
$u(t,0)=u(t,L)=0$.
Thus $u$ is a weak solution of \eqref{eq:nswe} in the sense of
Definition~\ref{def:weak-wave}.
\end{proof}




\subsection{Moment estimates and quasi-sure H\"older regularity}

For $\gamma\in(0,1]$, write
\[
C_0^\gamma([0,L])
:=
\left\{
f\in C([0,L]):
f(0)=f(L)=0,\ 
\sup_{x\ne y}
\frac{|f(x)-f(y)|}{|x-y|^\gamma}<\infty
\right\}.
\]

\begin{lemma}
\label{lem:I0-regularity}
Assume $u_0\in C_0^\gamma([0,L])$ and $v_0\in L^\infty(0,L)$. Then, for
every $T>0$, there exists $C=C(T,L,u_0,v_0,\gamma)$ such that
\begin{equation}
\label{eq:I0-holder}
|I_0(t,x)-I_0(s,y)|
\le
C\bigl(|t-s|^\gamma+|x-y|^\gamma\bigr),
\qquad (t,x),(s,y)\in[0,T]\times[0,L].
\end{equation}
\end{lemma}

\begin{proof}
The odd $2L$-periodic extension $u_0^{\mathrm{op}}$ is $\gamma$-H\"older
on bounded intervals. Since $v_0\in L^\infty(0,L)$, its odd $2L$-periodic extension
$v_0^{\mathrm{op}}$ belongs to $L^\infty(\mathbb R)$. Hence
\[
F(r):=\int_0^r v_0^{\mathrm{op}}(q)\,dq
\]
is globally Lipschitz, with
\[
|F(r)-F(r')|\le
\|v_0\|_{L^\infty(0,L)}|r-r'|,
\qquad r,r'\in\mathbb R.
\]
Formula \eqref{eq:I0} can be written as
\[
I_0(t,x)
=
\frac12\left(u_0^{\mathrm{op}}(x+t)+u_0^{\mathrm{op}}(x-t)\right)
+
\frac12\bigl(F(x+t)-F(x-t)\bigr).
\]
The first term is jointly $\gamma$-H\"older in $(t,x)$, while the second is
jointly Lipschitz. On a compact cylinder, every Lipschitz function is also
$\gamma$-H\"older for $0<\gamma\le1$, which proves \eqref{eq:I0-holder}.
\end{proof}

\begin{theorem}
\label{thm:wave-moment}
Assume \eqref{eq:Lipschitz-ab},
$u_0\in C_0^\gamma([0,L])$ and $v_0\in L^\infty(0,L)$ for some
$\gamma\in(0,1]$, and let $u$ be the global mild solution of
\eqref{eq:nswe}. Set
\[
\theta:=\gamma\wedge\frac12.
\]
Then, for every $T>0$ and $p\ge2$, there is a constant
$C=C(p,T,L,a,b,u_0,v_0,\overline\sigma)<\infty$ such that
\begin{equation}
\label{eq:u-joint-moment}
\hat{\mathbb E}
\left[|u(t,x)-u(s,y)|^p\right]
\le
C
\left(
|t-s|^\theta+|x-y|^\theta
\right)^p
\end{equation}
for all $(t,x),(s,y)\in[0,T]\times[0,L]$.
\end{theorem}

\begin{proof}
By Theorem~\ref{th:main},
$u\in\mathbf S_G^p([0,T]\times[0,L])$ for every $p\ge2$. Since $a$ and $b$ are globally Lipschitz, composition with $a$ and $b$
preserves $\mathbf S_G^p([0,T]\times[0,L])$. Hence
$a\circ u,\ b\circ u
\in\mathbf S_G^p([0,T]\times[0,L])$.
From \eqref{eq:nswe-df},
Lemma~\ref{lem:I0-regularity}, and
Proposition~\ref{prop:wave-convolution},
\begin{align*}
\|u(t,x)-u(s,y)\|_{\mathbf L_G^p}
&\le
C\bigl(|t-s|^\gamma+|x-y|^\gamma\bigr)
+C\bigl(|t-s|+|x-y|\bigr)\\
&\quad+
C\bigl(|t-s|^{1/2}+|x-y|^{1/2}\bigr).
\end{align*}
On the compact cylinder, the right-hand side  is bounded by
$C(|t-s|^\theta+|x-y|^\theta)$, and raising to the $p$th power proves
\eqref{eq:u-joint-moment}.
\end{proof}

\begin{theorem}
\label{thm:Holder}
Under the assumptions of Theorem~\ref{thm:wave-moment}, the global mild
solution admits a modification $\widetilde u$ whose paths are locally
H\"older continuous on $[0,\infty)\times[0,L]$ quasi-surely. More precisely:
\begin{enumerate}
\item[(i)] If $\gamma\ge\tfrac12$, then, quasi-surely, on every compact
cylinder $[0,T]\times[0,L]$ the paths of $\widetilde u$ are jointly
H\"older continuous of every order $\alpha\in(0,\tfrac12)$. 
\item[(ii)] If $0<\gamma<\tfrac12$, then, quasi-surely, the paths of
$\widetilde u$ are jointly $\gamma$-H\"older continuous on every compact
cylinder. 
\end{enumerate}
\end{theorem}

\begin{proof}
Fix $T>0$. Suppose first that $\gamma\ge1/2$. Let $\alpha<1/2$ and choose
$p>4$ so large that
$\alpha<1/2-2/p$. By \eqref{eq:u-joint-moment},
\[
\hat{\mathbb E}[|u(t,x)-u(s,y)|^p]
\le
C(|t-s|+|x-y|)^{p/2}.
\]
Applying Kolmogorov  criterion under sublinear expectation \cite[Theorem 6.1.40]{P2019} with
$2+\varepsilon=p/2$ gives an $\alpha$-H\"older modification on
$[0,T]\times[0,L]$. Taking a countable sequence
$\alpha_k\uparrow1/2$ and identifying the corresponding continuous
modifications on the countable dense set of rational points yields one
modification possessing every exponent $\alpha<1/2$ outside a single polar
set.

Now suppose $0<\gamma<1/2$. Write
\[
u=I_0+V+Z,
\qquad
V:=\mathcal D(b\circ u),
\qquad
Z:=\mathcal S(a\circ u).
\]
By Lemma~\ref{lem:I0-regularity}, $I_0$ is $\gamma$-H\"older. Choose
$\beta_Z$ with $\gamma<\beta_Z<1/2$. By \eqref{eq:S-joint-inc},
\[
\hat{\mathbb E}[|Z(t,x)-Z(s,y)|^p]
\le
C_p(|t-s|+|x-y|)^{p/2}.
\]
Choose $p$ large enough that $\beta_Z<1/2-2/p$. Then the Kolmogorov criterion
gives a $\beta_Z$-H\"older modification
of $Z$. Similarly, choose $\beta_V$ with $\gamma<\beta_V<1$. From
\eqref{eq:D-joint-inc},
\[
\hat{\mathbb E}[|V(t,x)-V(s,y)|^p]
\le
C_p(|t-s|+|x-y|)^p.
\]
Taking $p$ sufficiently large that $\beta_V<1-2/p$ and applying
the Kolmogorov criterion gives a $\beta_V$-H\"older modification
of $V$. Hence
\[
\widetilde u:=I_0+\widetilde V+\widetilde Z
\]
is a modification of $u$ and is jointly $\gamma$-H\"older on the compact
cylinder, because $\beta_Z,\beta_V>\gamma$. This is why the endpoint
$\gamma$ is attained when $\gamma<1/2$.

Finally apply the preceding construction on
$[0,n]\times[0,L]$, $n\in\mathbb N$. On each overlap, the modifications for
horizons $n$ and $n+1$ agree quasi-surely at every rational point because both
are modifications of the same random field. Removing the countable union of
the corresponding polar sets and using continuity shows that they agree on the
whole overlap. Hence the finite-horizon modifications patch together to a
single global modification $\widetilde u$ with the asserted local H\"older
properties.
\end{proof}

\end{document}